\newcount\Comments  
\documentclass[letterpaper,10pt]{scrartcl}

\usepackage[utf8]{inputenc}
\usepackage[english]{babel}
\usepackage[automark]{scrpage2}
\usepackage{amsmath,amstext,amsbsy,amsopn,amscd,amsxtra,upref,amssymb}
\usepackage{amsfonts,bm}
\usepackage{amsthm}
\usepackage{color}
\usepackage{graphicx}
\usepackage{sidecap}
\usepackage{multirow}
\usepackage{booktabs}
\usepackage{bm}
\usepackage{fullpage}

\usepackage{tabularx}
\usepackage[font=small]{caption}
\usepackage{subfig}

\usepackage{longtable}
\usepackage{rotating}

\usepackage{algorithm}
\usepackage{algpseudocode}
\extrafloats{100}

\usepackage{lineno,hyperref}

\graphicspath{{figures/}}
\DeclareGraphicsExtensions{.pdf,.eps,.png,.jpg,.jpeg}

\newcommand{\bfx}{\boldsymbol x}
\newcommand{\bfX}{\boldsymbol X}

\newcommand{\bff}{\boldsymbol f}

\newcommand{\Dcal}{\mathcal{D}}

\newcommand{\Vcal}{\mathcal{V}}

\newcommand{\bfmu}{\boldsymbol \mu}

\newcommand{\bfu}{\boldsymbol u}
\newcommand{\bfv}{\boldsymbol v}

\newcommand{\bfA}{\boldsymbol A}
\newcommand{\bfB}{\boldsymbol B}

\newcommand{\bfD}{\boldsymbol D}

\newcommand{\bfO}{\boldsymbol O}

\newcommand{\bfU}{\boldsymbol U}
\newcommand{\bfV}{\boldsymbol V}

\newcommand{\bfY}{\boldsymbol Y}

\newcommand{\nh}{N}
\newcommand{\nr}{n}

\newcommand{\hbfX}{\hat{\bfX}}
\newcommand{\bbfX}{\breve{\bfX}}
\newcommand{\hbfA}{\hat{\bfA}}
\newcommand{\hbfB}{\hat{\bfB}}

\newcommand{\hbfx}{\hat{\bfx}}
\newcommand{\rbfx}{\bar{\bfx}}
\newcommand{\bbfx}{\breve{\bfx}}
\newcommand{\tbfA}{\tilde{\bfA}}
\newcommand{\tbfB}{\tilde{\bfB}}
\newcommand{\tbfx}{\tilde{\bfx}}
\newcommand{\tbfX}{\tilde{\bfX}}
\newcommand{\rbfX}{\bar{\bfX}}

\newcommand{\tbfO}{\tilde{\bfO}}
\newcommand{\hbfO}{\hat{\bfO}}

\newcommand{\tbfY}{\tilde{\bfY}}
\newcommand{\bbfY}{\breve{\bfY}}
\newcommand{\rbfY}{\bar{\bfY}}

\newcommand{\bbfD}{\breve{\bfD}}
\newcommand{\rbfD}{\bar{\bfD}}

\newcommand{\sfA}{\mathsf{A}}
\newcommand{\tsfA}{\tilde{\sfA}}

\newcommand{\bbff}{\breve{\bff}}

\newcommand{\bfZ}{\boldsymbol{Z}}
\newcommand{\bfxi}{\boldsymbol{\xi}}

\newcommand{\hbff}{\hat{\bff}}
\newcommand{\tbff}{\tilde{\bff}}
\newcommand{\reprojnr}{\bar{\nr}}

\newcommand{\kibitz}[2]{\ifnum\Comments=1\textcolor{#1}{#2}\fi}

\newenvironment{keywords}%
   {\begin{trivlist}\item[]{\bfseries\sffamily Keywords:}\ }
   {\end{trivlist}}

\ihead[]{}
\ohead[]{}
\ifoot[]{}
\ofoot[]{}

\newtheorem{corollary}{Corollary}
\theoremstyle{definition}

\newtheorem{proposition}{Proposition}

\ifpdf
\hypersetup{
  pdftitle={Sampling low-dimensional Markovian dynamics for pre-asymptotically recovering reduced models from data with operator inference},
  pdfauthor={Benjamin Peherstorfer}
}
\fi

\title{Sampling low-dimensional Markovian dynamics for pre-asymptotically recovering reduced models from data with operator inference}

\author{Benjamin Peherstorfer\thanks{Courant Institute of Mathematical Sciences, New York University, New York, NY 10012}}

\begin{document}

\maketitle

\begin{abstract}
This work introduces a method for learning low-dimensional models from data of high-dimensional black-box dynamical systems. The novelty is that the learned models are exactly the reduced models that are traditionally constructed with model reduction techniques that require full knowledge of governing equations and operators of the high-dimensional systems. Thus, the learned models are guaranteed to inherit the well-studied properties of reduced models from traditional model reduction. The key ingredient is a new data sampling scheme to obtain re-projected trajectories of high-dimensional systems that correspond to Markovian dynamics in low-dimensional subspaces. The exact recovery of reduced models from these re-projected trajectories is guaranteed pre-asymptotically under certain conditions for finite amounts of data and for a large class of systems with polynomial nonlinear terms. Numerical results demonstrate that the low-dimensional models learned with the proposed approach match reduced models from traditional model reduction up to numerical errors in practice. The numerical results further indicate that low-dimensional models fitted to re-projected trajectories are predictive even in situations where models fitted to trajectories without re-projection are inaccurate and unstable. 
\end{abstract}


\section{Introduction}
\label{sec:Intro}
Reduced models have become a ubiquitous tool to make tractable computations that require large numbers of model evaluations in, e.g., uncertainty quantification, optimization, and inverse problems. 
Traditional model reduction derives reduced models from high-dimensional (full) models of systems that typically are given in the form of partial differential equations (PDEs) and their corresponding discretized operators. The properties of reduced models have been extensively studied by the model reduction community \cite{AntoulasBook,RozzaPateraSurvey,AntBG10,MORSurveySIREV} and even rigorous error estimation has been established for certain classes of problems \cite{veroy_posteriori_2003,doi:10.1002/fld.867,refId0,doi:10.1080/13873954.2010.514703,HaasdonkError,RozzaPateraSurvey}. 
The aim of data-driven model reduction methods is to learn reduced models from data alone and so to extend the scope of model reduction to settings where the governing equations and the corresponding discrete operators of the high-dimensional systems are unavailable; however, the models learned from data alone typically are only approximations of the reduced models obtained with traditional model reduction and thus establishing the same rigor for the learned models as for reduced models is challenging.  
In contrast, this work presents an approach to learn low-dimensional models from data that exactly match the reduced models that are obtained with traditional model reduction as if the governing equations and discrete operators of the high-dimensional systems were available. This guarantee of exactly recovering reduced models from data holds pre-asymptotically in the number of data points and for a wide class of high-dimensional systems with polynomial nonlinear terms under certain conditions. Thus, models learned with the proposed approach \emph{are} the reduced models of traditional model reduction and therefore directly inherit their well-studied properties. 

There is a large body of literature on learning dynamical-system models from data. We review only the works that are most relevant for the proposed approach. First, there is system identification that originated in the systems and control community \cite{LjungBook}. The Loewner approach was introduced by Antoulas and collaborators \cite{ANTOULAS01011986,5356286,Mayo2007634,BeaG12} and has been extended from linear time-invariant systems to parametrized \cite{doi:10.1137/130914619}, bilinear \cite{doi:10.1137/15M1041432} and quadratic-bilinear systems \cite{doi:10.1002/nla.2200}. Under certain conditions, the models learned with the Loewner approach are the reduced models that are obtained with interpolatory model reduction; however, Loewner models are learned from frequency-response data rather than from time-domain data. The work \cite{PSW16TLoewner} builds on Loewner to learn reduced models of linear time-invariant systems from time-domain data; however, learning from time-domain data can introduce errors and so the learned models can differ from the corresponding Loewner models derived from frequency-response data. Second, there is dynamic mode decomposition \cite{SchmidDMD,FLM:7843190,FLM:6837872,Tu2014391,NathanBook} that best-fits linear operators to state trajectories with respect to the $L_2$ norm. Methods based on the Koopman operator have been developed as one path to extending dynamic mode decomposition to nonlinear dynamical systems \cite{Mezic2005,Williams2015,NathanKoopman1}. Third, there are methods that learn parsimonious models by exploiting sparsity in the high-dimensional systems, e.g., the work by Schaeffer and collaborators \cite{Schaeffer6634,doi:10.1137/18M116798X} and the work by Kutz, Brunton, and collaborators \cite{Brunton12042016,Rudye1602614}. The learned models typically are either continuous in the sense that terms of PDEs are learned from a dictionary or high-dimensional models are learned that inherit sparsity from, e.g., finite-element discretizations of the governing equations of the systems of interest. In contrast, we aim to learn low-dimensional models that help to reduce computational costs in applications that require many model evaluations \cite{MORSurveySIREV,PWG17MultiSurvey}. 

Instead of aiming to find models that best-fit data, we aim to exactly recover reduced models from data so that our models inherit the reduced models' well-studied properties. Our approach is based on operator inference \cite{Peherstorfer16DataDriven}, which has been derived from \cite{pehersto15dynamic} and is a data-driven model reduction approach that learns approximations of reduced models from state trajectories. In \cite{Peherstorfer16DataDriven}, operator inference has been introduced for systems with polynomial nonlinear terms and in \cite{doi:10.2514/6.2019-3707} operator inference is combined with the transform \& learn approach to obtain models of systems with more general nonlinear terms. Operator inference projects trajectories of systems of interest onto low-dimensional subspaces of the high-dimensional state spaces and then fits operators to the projected trajectories via least-squares regression. However, as is known from, e.g., the Mori-Zwanzig formalism from statistical physics \cite{Givon_2004,chorin2006}, the projected trajectories correspond to non-Markovian dynamics in the low-dimensional subspaces even though the high-dimensional trajectories and the corresponding high-dimensional systems are Markovian. The non-Markovian dynamics are related to the closure error in model reduction \cite{Wang12PODclosureComp,2017arXiv171003569F,doi:10.1137/18M1177263,doi:10.1098/rspa.2017.0385,doi:10.1098/rspa.2014.0446,doi:10.1063/1.5003467}. 
 To account for the non-Markovian dynamics, methods have been proposed that learn non-Markovian terms \cite{CHORIN2002239,doi:10.1137/18M1177263,doi:10.1098/rspa.2017.0385,doi:10.1063/1.5003467} and that use time-delay and other embeddings \cite{doi:10.1063/1.5063730,doi:10.1137/15M1054924,doi:10.1137/18M1188227}; however, since we aim to exactly recover the Markovian reduced models that are obtained with traditional model reduction, neither of these remedies are applicable in our situation. Instead, 
we propose a data sampling scheme that iterates between time stepping the high-dimensional systems and projecting onto low-dimensional subspaces to generate trajectories that correspond to low-dimensional Markovian dynamics. We then show that, under certain conditions, applying operator inference to these re-projected trajectories gives the same operators that are obtained with traditional model reduction methods. The result is a pre-asymptotic guarantee to exactly recover reduced models from finite amounts of data for a wide class of systems with polynomial nonlinear terms. Our numerical results demonstrate these theoretical results in practice by learning low-dimensional models that match the reduced models from traditional model reduction up to numerical errors.

Section~\ref{sec:Prelim} discusses preliminaries on dynamical systems, traditional model reduction, operator inference, and formulates the problem. Section~\ref{sec:ReProj} introduces data sampling with re-projection to obtain trajectories that correspond to low-dimensional Markovian dynamics and provides an analysis that shows that operators fitted to these re-projected trajectories are the operators obtained with traditional model reduction. The overall computational approach is presented in Algorithm~\ref{alg:OpInfWithReProj} in Section~\ref{sec:CompProc} and numerical results are given in Section~\ref{sec:NumExp}. Conclusions are drawn in Section~\ref{sec:Conc}.

\section{Preliminaries}
\label{sec:Prelim}
The focus of this work is on dynamical systems with polynomial nonlinear terms, which we introduce in Section~\ref{sec:Prelim:FOM} together with traditional model reduction for these systems in Section~\ref{sec:Prelim:MOR}. A building block of our approach is operator inference \cite{Peherstorfer16DataDriven} for learning reduced models from data, which we discuss in Section~\ref{sec:Prelim:OpInf}. The problem we aim to address is formulated in Section~\ref{sec:Prelim:ProbForm}.

\subsection{Dynamical systems with polynomial nonlinear terms}
\label{sec:Prelim:FOM}
Let $K \in \mathbb{N}$ and consider a dynamical system of the form
\begin{equation}
\bfx_{k + 1}(\bfmu) = \bff(\bfx_k(\bfmu), \bfu_k(\bfmu); \bfmu)\,,\qquad k = 0, \dots, K-1\,,
\label{eq:Prelim:FOMGeneral}
\end{equation}
with state $\bfx_k(\bfmu) \in \mathbb{R}^{\nh}$ of dimension $\nh \in \mathbb{N}$ and input $\bfu_k(\bfmu) \in \mathbb{R}^p$ of dimension $p \in \mathbb{N}$ at time steps $k = 0, \dots, K-1$. The parameter $\bfmu \in \Dcal \subset \mathbb{R}^{d}$ is independent of the time step. The initial condition is $\bfx_0 \in \mathbb{R}^{\nh}$. The potentially nonlinear function $\bff: \mathbb{R}^{\nh} \times \mathbb{R}^p \times \Dcal \to \mathbb{R}^{\nh}$ describes the dynamics of system \eqref{eq:Prelim:FOMGeneral}. Set $\nh_i = {\nh + i - 1 \choose i}$ for $i \in \mathbb{N}$. In the following, we consider systems \eqref{eq:Prelim:FOMGeneral} that are polynomial of order $\ell \in \mathbb{N}$, which means that there exists $\bfA_i(\bfmu) \in \mathbb{R}^{\nh \times \nh_i}$ for $i = 1, \dots, \ell$ and $\bfB(\bfmu) \in \mathbb{R}^{\nh \times p}$ for $\bfmu \in \Dcal$ such that
\begin{equation}
\bff(\bfx_k(\bfmu), \bfu_k(\bfmu); \bfmu) = \sum_{i = 1}^{\ell} \bfA_i(\bfmu)\bfx_k^i(\bfmu) + \bfB(\bfmu)\bfu_k(\bfmu)\,,\qquad k = 0, \dots, K-1\,.
\label{eq:Prelim:FOMPolynomial}
\end{equation}
The vector $\bfx^i_k(\bfmu) \in \mathbb{R}^{\nh_i}$ is the $i$-th power of $\bfx_k$, which is constructed from the Kronecker product $\bfx_k(\bfmu) \otimes \dots \otimes \bfx_k(\bfmu)$ by removing all duplicate entries due to commutativity of the multiplication \cite{Peherstorfer16DataDriven}. Note that $\bfx_k^1(\bfmu) = \bfx_k(\bfmu)$. Define the trajectories $\bfX(\bfmu) = [\bfx_0(\bfmu), \dots, \bfx_{K-1}(\bfmu)] \in \mathbb{R}^{\nh \times K}$ and $\bfY(\bfmu) = [\bfx_1(\bfmu), \dots, \bfx_K(\bfmu)] \in \mathbb{R}^{\nh \times K}$, which differ in their start and end index. Let further $\bfX^i(\bfmu) = [\bfx_0^i(\bfmu), \dots, \bfx_{K - 1}^i(\bfmu)] \in \mathbb{R}^{\nh \times K}$, for $i = 1, \dots, \ell$, be the trajectories corresponding to the $i$-th powers of the states at times $k = 0, \dots, K-1$.  More details on systems with polynomial nonlinear terms and their relevance in computational science and engineering can be found in, e.g., \cite{5991229,doi:10.1137/16M1098280,BorisLift,7330699,Peherstorfer16DataDriven}.

\subsection{Model reduction of systems with polynomial nonlinear terms}
\label{sec:Prelim:MOR}
If operators $\bfA_1(\bfmu), \dots, \bfA_{\ell}(\bfmu), \bfB(\bfmu)$ of \eqref{eq:Prelim:FOMGeneral} for $\bfmu \in \Dcal$ are available, then traditional projection-based model reduction can be applied to find a reduced model; see, e.g., \cite{RozzaPateraSurvey,MORSurveySIREV}. Traditional projection-based model reduction typically first constructs a reduced space and then projects the operators of the high-dimensional system to obtain the reduced operators and to assemble the reduced model.  
Consider first the construction of a reduced space. Let $\bfmu_1, \dots, \bfmu_m \in \Dcal$ and let $\bfX(\bfmu_1), \dots, \bfX(\bfmu_m) \in \mathbb{R}^{\nh \times K}$ be the corresponding trajectories of length $K$. Applying proper orthogonal decomposition (POD) \cite{MORSurveySIREV,SirovichMethodOfSnapshots} to the snapshot matrix $[\bfX(\bfmu_1), \dots, \bfX(\bfmu_m)] \in \mathbb{R}^{\nh \times mK}$ yields an orthonormal basis $\bfv_1, \dots, \bfv_{\nr}$, with $\nr \ll \nh$, that spans an $\nr$-dimensional subspace $\Vcal_{\nr} \subset \mathbb{R}^{\nh}$. Let $\bfV_{\nr} = [\bfv_1, \dots, \bfv_{\nr}] \in \mathbb{R}^{\nh \times \nr}$ be the basis matrix that has as columns the basis vectors $\bfv_1, \dots, \bfv_n$. Note that $\bfV_{\nr}$ is independent of the parameter $\bfmu$ in the following. There are other methods for constructing reduced spaces such as greedy methods \cite{prudhomme_reliable_2001,veroy_posteriori_2003} and interpolatory model reduction \cite{Antoulas2010,gugercin_2008,AntoulasBook}. We refer to \cite{RozzaPateraSurvey,MORSurveySIREV} for details on how to select the parameters $\bfmu_1, \dots, \bfmu_m$ and how to select the dimension $\nr$ of the space $\Vcal_{\nr}$. 

For $j = 1, \dots, m$, the reduced operators are constructed via, e.g., Galerkin projection
\begin{equation}
\tbfA_1(\bfmu_j) = \bfV_{\nr}^T\bfA_1(\bfmu_j)\bfV_{\nr}\,,\qquad \tbfB(\bfmu_j) = \bfV_{\nr}^T\bfB(\bfmu_j)\,,
\label{eq:Prelim:ROMOperators}
\end{equation}
and similarly for $\tbfA_2(\bfmu_j) \in \mathbb{R}^{\nr \times \nr_2}, \dots, \tbfA_{\ell}(\bfmu_j) \in \mathbb{R}^{\nr \times \nr_{\ell}}$ with
\begin{equation}
\nr_i = {\nr + i - 1 \choose i}\,, \qquad i \in \mathbb{N}\,.
\label{eq:Prelim:nri}
\end{equation} 
The reduced model for $\bfmu_j$ is
\begin{equation}
\begin{aligned}
\tbfx_{k + 1}(\bfmu_j) = & \tbff(\tbfx_{k}(\bfmu_j), \bfu_k(\bfmu_j); \bfmu_j) \\
 = & \sum_{i = 1}^{\ell}\tbfA_i(\bfmu_j)\tbfx_k^i(\bfmu_j) +\tbfB(\bfmu_j) \bfu_k(\bfmu_j)\,,\qquad k = 0, \dots, K-1\,,
\end{aligned}
\label{eq:Prelim:iROM}
\end{equation}
with the reduced state $\tilde{\bfx}_k(\bfmu_j) \in \mathbb{R}^{\nr}$ and its $i$-th power $\tilde{\bfx}_k^i(\bfmu_j) \in \mathbb{R}^{\nr_i}$ for $i \in \mathbb{N}$. The initial condition is $\tbfx_0(\bfmu_j) = \bfV_{\nr}^T\bfx_0(\bfmu_j)$. Once the reduced models $\tbff(\cdot, \cdot; \bfmu_1), \dots, \tbff(\cdot, \cdot; \bfmu_m)$ are constructed for all $m$ parameters $\bfmu_1, \dots, \bfmu_m$, a reduced model for $\bfmu \in \Dcal$ is derived by element-wise interpolation of the reduced operators corresponding to $\bfmu_1, \dots, \bfmu_m$. If structure of the reduced operators is known, e.g., symmetry and positive definiteness, then this structure can preserved in the interpolation. We refer to \cite{NME:NME2681,panzer_parametric_2010,degroote_interpolation_2010} for details on interpolating reduced operators in model reduction.

\subsection{Operator inference}
\label{sec:Prelim:OpInf}
The traditional model reduction approach described in Section~\ref{sec:Prelim:MOR} to construct a reduced model \eqref{eq:Prelim:iROM} is intrusive in the sense that the operators $\bfA_1(\bfmu_j), \dots, \bfA_{\ell}(\bfmu_j), \bfB(\bfmu_j)$ of system \eqref{eq:Prelim:FOMPolynomial} for $j = 1, \dots, m$ are required in the projection step \eqref{eq:Prelim:ROMOperators}. Operator inference is introduced in \cite{Peherstorfer16DataDriven} to derive approximations of the reduced operators $\tbfA_1(\bfmu_j), \dots, \tbfA_{\ell}(\bfmu_j), \tbfB(\bfmu_j)$ from data of the high-dimensional system without requiring the high-dimensional operators $\bfA_1(\bfmu_j), \dots, \bfA_{\ell}(\bfmu_j), \bfB(\bfmu_j)$. 

\subsubsection{Operator inference} Operator inference proceeds in three steps. First, state trajectories $\bfX(\bfmu_1), \dots, \bfX(\bfmu_m)$ and $\bfY(\bfmu_1), \dots, \bfY(\bfmu_m)$ are obtained by querying the system \eqref{eq:Prelim:FOMGeneral} at parameters $\bfmu_1, \dots, \bfmu_m \in \Dcal$ to derive a reduced space spanned by the columns of $\bfV_{\nr} = [\bfv_1, \dots, \bfv_{\nr}]$. Many of the basis construction techniques developed in traditional model reduction can be applied; see references given in Section~\ref{sec:Prelim:MOR}. In the following, we will use POD to construct $\bfV_{\nr}$ as described in Section~\ref{sec:Prelim:MOR}. The second step of operator inference is to project the trajectories onto the reduced space $\Vcal_{\nr}$ spanned by the columns of $\bfV_{\nr}$ and so to obtain the projected trajectories 
\[
\bbfX(\bfmu_j) = \bfV_{\nr}^T\bfX(\bfmu_j)\,,\qquad \bbfY(\bfmu_j) = \bfV_{\nr}^T\bfY(\bfmu_j)\,,\qquad j = 1, \dots, m\,.
\]
In the third step of operator inference, the operators
\begin{equation}
\hbfA_1(\bfmu_j) \in \mathbb{R}^{\nr \times \nr_1}, \dots, \hbfA_{\ell}(\bfmu_j) \in \mathbb{R}^{\nr \times \nr_{\ell}}, \hbfB(\bfmu_j) \in \mathbb{R}^{\nr \times p}
\label{eq:Prelim:InferredOperators}
\end{equation}
are learned via least-squares regression 
\begin{equation}
\min_{\hbfA_1(\bfmu_j), \dots, \hbfA_{\ell}(\bfmu_j), \hbfB(\bfmu_j)} \sum_{k = 0}^{K-1} \left\| \sum_{i = 1}^{\ell} \hbfA_i(\bfmu_j)\bbfx_k^{i}(\bfmu_j) + \hbfB(\bfmu_j)\bfu_k(\bfmu_j) - \bbfx_{k + 1}(\bfmu_j) \right\|_2^2
\label{eq:Prelim:OpInfOpti}
\end{equation}
to obtain the model
\begin{equation}
\begin{aligned}
\hbfx_{k + 1}(\bfmu_j) = & \sum_{i = 1}^{\ell} \hbfA_i(\bfmu_j)\hbfx_k^{i}(\bfmu_j) + \hbfB(\bfmu_j)\bfu_k(\bfmu_j)\\
= & \hbff(\hbfx_k(\bfmu_j), \bfu_k(\bfmu_j); \bfmu_j)
\end{aligned}
\label{eq:Prelim:nROM}
\end{equation}
for $j = 1, \dots, m$. Note that the least-squares problem \eqref{eq:Prelim:OpInfOpti} is solved for each parameter $\bfmu_j$ with $j = 1, \dots, m$. The state of the learned model at time $k$ is $\hbfx_k(\bfmu_j) \in \mathbb{R}^{\nr}$ with its $i$-th power $\hbfx_k^i(\bfmu_j)$. Note that the state $\hbfx_k(\bfmu_j)$ is obtained by time stepping the learned model \eqref{eq:Prelim:nROM}, whereas the projected state $\bbfx_k(\bfmu_j)$ is obtained by projecting the high-dimensional state $\bfx_k(\bfmu_j)$ at time $k$ onto the reduced space $\Vcal_{\nr}$. The initial condition is $\hbfx_0(\bfmu_j) = \bfV_{\nr}^T\bfx_0(\bfmu_j)$. To obtain a model for $\bfmu \in \Dcal$, the operators of the learned models corresponding to $\bfmu_1, \dots, \bfmu_m$ are interpolated as in traditional model reduction; see Section~\ref{sec:Prelim:MOR}. We refer to \cite{Peherstorfer16DataDriven,doi:10.2514/6.2019-3707,2019arXiv190803620S,pehersto15dynamic} for details on operator inference.

\subsubsection{Data matrix} It will be convenient to write \eqref{eq:Prelim:OpInfOpti} for each $j = 1, \dots, m$ as
\begin{equation}
\min_{\hbfO(\bfmu_j)} \left\|\bbfD^T(\bfmu_j)\hbfO^T(\bfmu_j) - \bbfY^T(\bfmu_j)\right\|_F^2
\label{eq:Prelim:OpInfDOY}
\end{equation}
with the data matrix 
\begin{equation}
\bbfD(\bfmu_j) = 
\begin{bmatrix}
\bbfX(\bfmu_j)\\
\bbfX^{2}(\bfmu_j)\\
\vdots\\
\bbfX^{\ell}(\bfmu_j)\\
\bfU(\bfmu_j)
\end{bmatrix} \in \mathbb{R}^{\sum_{i = 1}^{\ell} n_i + p \times K}
\label{eq:Prelim:OpInfDataMatrixProj}
\end{equation}
and $\bbfX^i(\bfmu_j) = [\bbfx_0^i(\bfmu_j), \dots, \bbfx_{K - 1}^i(\bfmu_j)] \in \mathbb{R}^{\nr_i \times K}$ for $i = 1, \dots, \ell$ and $\bfU(\bfmu_j) = [\bfu_0(\bfmu_j), \dots, \bfu_{K - 1}(\bfmu_j)] \in \mathbb{R}^{p \times K}$. 
The operator matrix is
\[
\hbfO(\bfmu_j) = \begin{bmatrix}
\hbfA_1(\bfmu_j) & \dots & \hbfA_{\ell}(\bfmu_j) & \hbfB(\bfmu_j)
\end{bmatrix} \in \mathbb{R}^{\nr \times \sum_{i = 1}^{\ell} n_i + p}\,.
\]

\subsection{Problem formulation}
\label{sec:Prelim:ProbForm}
Our goal is exactly recovering the operators \eqref{eq:Prelim:ROMOperators} of the intrusive reduced model from data of the high-dimensional system without knowledge of the high-dimensional operators \eqref{eq:Prelim:ROMOperators}. The operators \eqref{eq:Prelim:InferredOperators} obtained with operator inference from the projected trajectories, as described in Section~\ref{sec:Prelim:OpInf}, equal the intrusive operators \eqref{eq:Prelim:ROMOperators} in the limit of $\nr \to \nh$ under certain conditions described in \cite{Peherstorfer16DataDriven}. However, typically, one is interested in reduced models with $\nr \ll \nh$, in which case the learned operators can differ significantly from the intrusive operators. To explain the origin of the difference between the intrusive and the learned, non-intrusive operators, consider the trajectory $\tbfX(\bfmu) = [\tbfx_0(\bfmu), \dots, \tbfx_{K - 1}(\bfmu)] \in \mathbb{R}^{\nr \times K}$ obtained by time stepping the intrusive reduced model \eqref{eq:Prelim:iROM}. Even if $\bfx_0(\bfmu) \in \Vcal_{\nr}$, and thus $\tbfx_0(\bfmu) = \bbfx_0(\bfmu)$, the projected trajectory $\bbfX(\bfmu)$ can be different from the intrusive trajectory $\tbfX(\bfmu)$, i.e., there is a non-zero closure error
\begin{equation}
\|\bbfX(\bfmu) - \tbfX(\bfmu)\|_F\,.
\label{eq:Prelim:ClosureError}
\end{equation}
By fitting operators to projected trajectories with operator inference as described in Section~\ref{sec:Prelim:OpInf} and in \cite{Peherstorfer16DataDriven}, the closure error \eqref{eq:Prelim:ClosureError} is introduced into the learned operators, which means that the learned operators can fail to approximate the dynamics of the intrusive reduced model. 

\begin{figure}
\centering
\begin{tabular}{cc}
{\resizebox{0.48\columnwidth}{!}{\LARGE\input{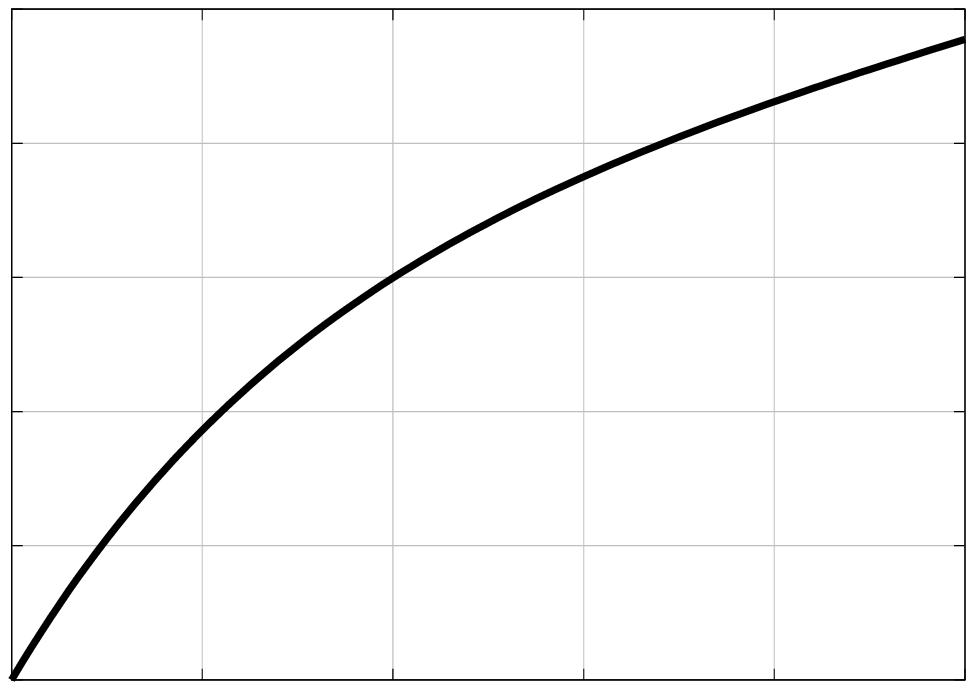}}} & {\resizebox{0.48\columnwidth}{!}{\LARGE\input{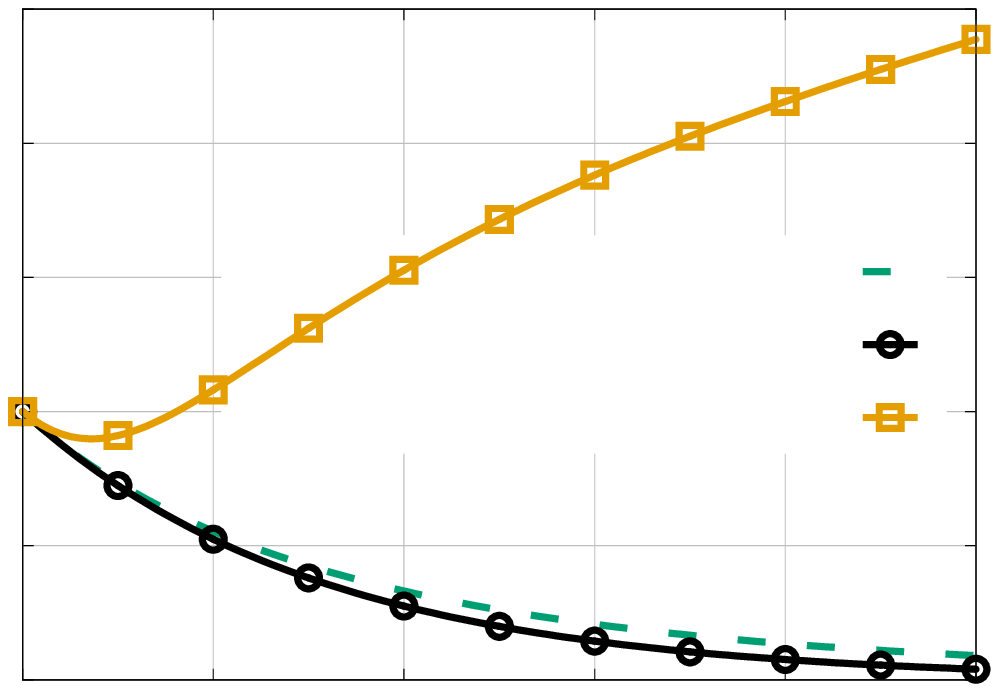}}}\\
\scriptsize (a) closure error & \scriptsize (b) trajectories
\end{tabular}
\caption{Toy example: The closure error \eqref{eq:Prelim:ClosureError} pollutes operators that are fitted to projected trajectories with operator inference, which can lead to models that only poorly approximate the corresponding intrusive reduced models and projected trajectories of the high-dimensional systems.}
\label{fig:Prelim:Example}
\end{figure}

We demonstrate the effect of the closure error on operator inference with a toy example. Consider a system \eqref{eq:Prelim:FOMPolynomial} of degree $\ell = 1$, order $\nh = 10$, time steps $K = 100$, and without inputs, i.e., a time-discrete autonomous linear dynamical system $\bfx_{k + 1} = \bfA_1\bfx_k$ for $k = 0, \dots, K - 1$. The matrix $\bfA_1 \in \mathbb{R}^{\nh \times \nh}$ is generated by first sampling entries uniformly in $[0, 1]$ and then transforming them to ensure the eigenvalues of $\bfA_1$ have absolute values less than 1. The initial condition $\bfx_0$ is the first column of the identity matrix of dimension $\nr \times \nr$ and $\bfX$ is the corresponding trajectory of length $K$. Set $\nr = 2$ and consider the 2-dimensional space $\Vcal_{\nr}$ that is spanned by the initial condition and the canonical unit vector with 1 at component 2. Let $\bfV_{\nr}$ be the corresponding basis matrix and let $\bbfX = \bfV_{\nr}^T\bfX$ be the projected trajectory. The intrusive reduced model is given by $\tbfx_{k + 1} = \tbfA_1\tbfx_k$ with $\tbfA_1 = \bfV_n^T\bfA_1\bfV_n$, $\tbfx_0 = \bfV_{\nr}^T\bfx_0 = [1, 0]^T$, and the trajectory $\tbfX$. Figure~\ref{fig:Prelim:Example}a shows the difference $\|\bbfx_k - \tbfx_k\|_2$ for time steps $k = 0, \dots, K-1$, which is the 2-norm of the difference of the projected state and the state of the intrusive reduced model at time step $k$. 
We now derive a model with operator inference from the projected trajectory $\bbfX$ as described in Section~\ref{sec:Prelim:OpInf} and denote the trajectory corresponding to this learned model as $\hbfX$. The trajectory $\hbfX$ differs significantly from the trajectory of the intrusive reduced model, as shown in Figure~\ref{fig:Prelim:Example}b. This toy example demonstrates that the closure error can have a significant polluting effect on fitting operators to projected trajectories and so lead to models that exhibit different dynamics than the corresponding intrusive reduced models and high-dimensional systems. Thus, if the aim is to learn from data the same reduced models that intrusive model reduction constructs, then there is a need for revising operator inference to guarantee the recovery of the intrusive operators from trajectories with finite length $K < \infty$ and for dimensions $\nr \ll \nh$.

\section{Sampling Markovian dynamics via re-projection}
\label{sec:ReProj}
We present a data sampling scheme that generates trajectories $\rbfX(\bfmu) \in \mathbb{R}^{\nr \times K}$ from the high-dimensional system \eqref{eq:Prelim:FOMGeneral} so that the closure error $\|\rbfX(\bfmu) - \tbfX(\bfmu)\|_F = 0$ is zero. Applying operator inference to these trajectories $\rbfX(\bfmu)$ of sufficient length $K < \infty$ exactly recovers the corresponding intrusive reduced model under certain conditions. 
In Section~\ref{sec:ReProj:MoriZwanzig}, we build on the Mori-Zwanzig formalism \cite{Givon_2004,chorin2006} to clarify that the closure error \eqref{eq:Prelim:ClosureError} corresponds to non-Markovian dynamics of the projected trajectories in $\Vcal_{\nr}$. These non-Markovian dynamics cannot be represented by a reduced model of the form \eqref{eq:Prelim:iROM}. Section~\ref{sec:ReProj:ReProj} describes a data sampling scheme that cancels these non-Markovian dynamics after each time step to obtain trajectories that correspond to Markovian dynamics in the reduced space $\Vcal_{\nr}$. Section~\ref{sec:ReProj:ExactRecovery} shows that these re-projected trajectories equal the trajectories that are obtained with an intrusive reduced model and thus that operator inference exactly recovers the intrusive reduced model from these re-projected trajectories. 

In this section, we focus on learning reduced models corresponding to a single parameter $\bfmu_j$, which then is subsequently repeated for all parameter $j = 1, \dots, m$. To ease exposition, we drop the dependence on $\bfmu_j$ in this section.

\subsection{Non-Markovian dynamics of projected states}
\label{sec:ReProj:MoriZwanzig}
To motivate our data sampling scheme, we first discuss the closure error  $\|\bbfX - \tbfX\|_F$ on the toy example given in the problem formulation in Section~\ref{sec:Prelim:ProbForm}. The arguments in this section are not new; we refer to the literature from the statistical physics community on the Mori-Zwanzig formalism, which describes the arguments in this section for more general systems and in stochastic settings; see, e.g., the surveys \cite{Givon_2004,chorin2006} for more details. 

Our toy example is an autonomous linear system, which corresponds to system \eqref{eq:Prelim:FOMGeneral} with $\ell = 1$ and $\bfB = \boldsymbol 0$, i.e.,
\begin{equation}
\bfx_{k + 1} = \bfA_1\bfx_k\,,\qquad k = 0, \dots, K-1\,.
\label{eq:ReProj:LinearSystem}
\end{equation}
Consider now the orthogonal complement $\Vcal_{\nr}^{\bot}$ of $\Vcal_{\nr}$ that is spanned by the orthonormal columns of the basis matrix $\bfV_{\nr}^{\bot} \in \mathbb{R}^{\nh \times \nh - \nr}$ such that $\mathbb{R}^{\nh} = \Vcal_{\nr} \oplus \Vcal_{\nr}^{\bot}$. Define $\bfx^{\parallel}_k = \bfV_{\nr}^T\bfx_k$ and $\bfx^{\bot}_k = (\bfV_{\nr}^{\bot})^T\bfx_k$ so that $\bfx_k = \bfV_{\nr}\bfx^{\parallel}_k + \bfV_{\nr}^{\bot}\bfx^{\bot}$. Correspondingly, \eqref{eq:ReProj:LinearSystem} is split into
\begin{align*}
\bfx_{k + 1}^{\parallel} & = \bfA_1^{\parallel\parallel}\bfx_k^{\parallel} + \bfA_1^{\parallel\bot}\bfx_k^{\bot}\\
\bfx_{k + 1}^{\bot} & = \bfA_1^{\bot\parallel}\bfx_k^{\parallel} + \bfA_1^{\bot\bot}\bfx_k^{\bot}\,,
\end{align*}
with the matrices
\[
\bfA_1^{\parallel\parallel} = \bfV_{\nr}^T\bfA_1\bfV_{\nr}\,, \bfA_1^{\parallel\bot} = \bfV_{\nr}^T\bfA_1\bfV_{\nr}^{\bot}\,, \bfA_1^{\bot\parallel} = (\bfV_{\nr}^{\bot})^T\bfA_1\bfV_{\nr}\,, \bfA_1^{\bot\bot} = (\bfV_{\nr}^{\bot})^T\bfA_1(\bfV_{\nr}^{\bot})^T\,.
\]
Model reduction as described in Section~\ref{sec:Prelim:MOR} constructs the reduced operator $\tbfA_1 = \bfA_1^{\parallel\parallel}$ via projection. Consider now the trajectory $\bfX = [\bfx_0, \bfx_1, \dots, \bfx_{K - 1}]$ and its projection $\bbfX = [\bbfx_0, \dots, \bbfx_{K - 1}]$ with $\bfV_{\nr}$. Then, we obtain 
\[
\bfV_{\nr}^T\bfx_{k + 1} = \bbfx_{k+1} = \bfx^{\parallel}_{k + 1} = \bfA_1^{\parallel\parallel}\bfx_k^{\parallel} + \bfA_1^{\parallel\bot}\bfx_k^{\bot}\,,\qquad k = 0, \dots, K-1\,,
\]
which gives with an inductive argument that 
\[
\bbfx_{k + 1} = \bfx_{k + 1}^{\parallel} = \underbrace{\bfA_1^{\parallel\parallel}\bfx_k^{\parallel}}_{\text{Markovian term}} + \underbrace{\bfA_1^{\parallel\bot}\sum_{i = 0}^{k-1} \left(\bfA_1^{\bot\bot}\right)^{k - 1 - i}\bfA_1^{\bot\parallel}\bfx_i^{\parallel}}_{\text{non-Markovian term}} + \bfA_1^{\parallel\bot}\left(\bfA_1^{\bot\bot}\right)^k\bfx_0^{\bot}\,.
\]
Thus, the projected state $\bbfx_{k + 1} = \bfx_{k + 1}^{\parallel}$ at time $k + 1$ depends on the time history of projected states $\bfx_0^{\parallel}, \bfx_1^{\parallel}, \dots, \bfx_k^{\parallel}$ instead of only on the last time step $\bfx_k^{\parallel}$. This means that the dynamics of the trajectory $\bfX$ become non-Markovian if projected onto the reduced space $\Vcal_{\nr}$ in the sense that going from $\bfx_k^{\parallel}$ to $\bfx_{k + 1}^{\parallel}$ requires knowledge of the time history $\bfx_0^{\parallel}, \dots, \bfx_{k-1}^{\parallel}$ in general. Therefore, the reduced model \eqref{eq:Prelim:iROM}, which is derived with traditional model reduction, cannot describe well the projected trajectory $\bbfX$ because the reduced model \eqref{eq:Prelim:iROM} is Markovian in the sense that the state $\tbfx_{k + 1}$ at time step $k + 1$ depends on the state $\tbfx_k$ of the previous time step $k$ alone, instead of on the history $\tbfx_{0}, \dots, \tbfx_{k-1}$.

\subsection{Data sampling with re-projection to avoid non-Markovian dynamics}
\label{sec:ReProj:ReProj}
We now describe our sampling scheme with re-projection. Consider an initial condition $\bfx_0 \in \Vcal_{\nr}$ and set $\rbfx_0 = \bfV_{\nr}^T\bfx_0$. Note that $\bfV_{\nr}\rbfx_0 = \bfx_0$ because $\bfx_0 \in \Vcal_{\nr}$. Our scheme proceeds iteratively, see Figure~\ref{fig:ReProj}. In the first iteration, system \eqref{eq:Prelim:FOMGeneral} is queried at initial condition $\bfV_{\nr}\rbfx_0$ and input $\bfu_0$ to obtain 
\[
\bfx_{\text{tmp}} = \bff(\bfV_{\nr}\rbfx_0, \bfu_0)\,.
\]
Then, the re-projected state $\rbfx_1 = \bfV_{\nr}^T\bfx_{\text{tmp}}$ is computed by projecting $\bfx_{\text{tmp}}$ onto $\Vcal_{\nr}$. In the second iteration, system \eqref{eq:Prelim:FOMGeneral} is queried for a single time step at the initial condition $\bfV_{\nr}\rbfx_1$ and input $\bfu_1$ to obtain $\bff(\bfV_{\nr}\rbfx_1, \bfu_1)$ and to compute $\rbfx_2$ via projection $\rbfx_2 = \bfV_n^T\bff(\bfV_n\rbfx_1, \bfu_1)$. This process is repeated to generate the re-projected states $\rbfx_0, \rbfx_1, \dots, \rbfx_K$ and to collect them into the re-projected trajectories $\rbfX = [\rbfx_0, \rbfx_1, \dots, \rbfx_{K - 1}]$ and $\rbfY = [\rbfx_1, \dots, \rbfx_K]$.

\begin{figure}
\centering
\fbox{\includegraphics[width=1.0\columnwidth]{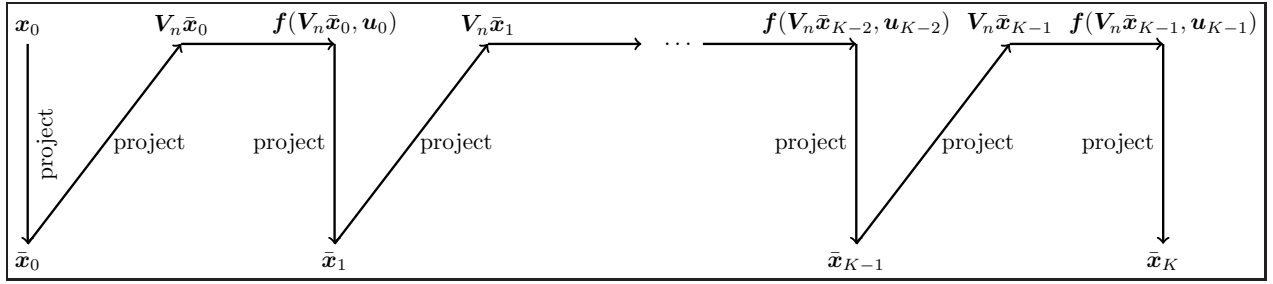}}
\caption{The scheme shows data sampling with re-projection. Under certain conditions that are discussed in Section~\ref{sec:ReProj:ExactRecovery}, the re-projected trajectories $\rbfX = [\rbfx_0, \rbfx_1, \dots, \rbfx_{K-1}]$ are the trajectories $\tbfX$ obtained by time stepping the intrusive reduced model. Thus, the closure error \eqref{eq:Prelim:ClosureError} of the re-projected trajectories is zero.}
\label{fig:ReProj}
\end{figure}

Algorithm~\ref{alg:ReProj} summarizes our data sampling scheme with re-projection. The inputs to Algorithm~\ref{alg:ReProj} are the high-dimensional system $\bff$, a basis matrix $\bfV_{\nr}$, an initial condition $\bfx_0 \in \Vcal_{\nr}$, a parameter $\bfmu \in \Dcal$, and inputs $\bfu_0, \dots, \bfu_{K - 1}$. Line~\ref{alg:ReProj:ICinV} projects the initial condition $\bfx_0$ to obtain $\rbfx_0$. The \texttt{for} loop on line~\ref{alg:ReProj:For} iterates over the time steps $k = 0, \dots, K-1$ and generates the re-projected state $\rbfx_{k + 1}$ by querying the high-dimensional system for a single time step in line~\ref{alg:ReProj:Query}. The re-projected trajectories $\rbfX$ and $\rbfY$ are returned in line~\ref{alg:ReProj:Return}.

\begin{algorithm}[t]
\caption{Data sampling with re-projection}\label{alg:ReProj}
\begin{algorithmic}[1]
\Procedure{ReProj}{$\bff, \bfV_{\nr}, \bfx_0, \bfu_0, \dots, \bfu_{K-1}$}
\State Set $\rbfx_0 = \bfV_{\nr}^T\bfx_0$\label{alg:ReProj:ICinV}
\For{$k = 0, \dots, K-1$}\label{alg:ReProj:For}
\State Query system for a single time step $\bfx_{\text{tmp}} = \bff(\bfV_{\nr}\rbfx_k, \bfu_k)$\label{alg:ReProj:Query}
\State Set $\rbfx_{k+1} = \bfV_{\nr}^T\bfx_{\text{tmp}}$\label{alg:ReProj:ReProj}
\EndFor
\State Return $\rbfX = [\rbfx_0, \rbfx_1, \dots, \rbfx_{K - 1}]$ and $\rbfY = [\rbfx_1, \dots, \rbfx_K]$\label{alg:ReProj:Return}
\EndProcedure
\end{algorithmic}
\end{algorithm}

\subsection{Exact recovery of reduced models from re-projected trajectories}
\label{sec:ReProj:ExactRecovery}
Proposition~\ref{prop:rX=tX} shows that the trajectories $\rbfX$ and $\rbfY$ obtained with sampling with re-projection are the trajectories $\tbfX$ and $\tbfY$ obtained from time stepping the corresponding intrusive reduced models. Proposition~\ref{prop:rX=tX} leads to Corollary~\ref{cor:Exact} that shows that intrusive reduced models are exactly recovered from re-projected trajectories in the sense that $\|\tbfA_i - \hbfA_i\|_F = \|\tbfB - \hbfB\|_F = 0$ for $i = 1, \dots, \ell$ under certain conditions. This is a pre-asymptotic result in the sense that it holds for finite-length trajectories, i.e., for finite number of data points, and for reduced spaces $\Vcal_{\nr}$ of dimensions $\nr < \nh$.

\begin{proposition}
Consider a system \eqref{eq:Prelim:FOMPolynomial} with polynomial nonlinear terms. Let $\bfx_0 \in \Vcal_{\nr}$ be an initial condition and let $\bfu_0, \dots, \bfu_{K-1}$ be inputs. Generate trajectories $\rbfX$ and $\rbfY$ from system \eqref{eq:Prelim:FOMPolynomial} with re-projection as described in Algorithm~\ref{alg:ReProj}. Then, $\rbfX = \tbfX$ and $\rbfY = \tbfY$ holds, where $\tbfX = [\tbfx_0, \dots, \tbfx_{K - 1}]$ and $\tbfY = [\tbfx_1, \dots, \tbfx_K]$ are trajectories obtained with the intrusive reduced model \eqref{eq:Prelim:iROM} with initial condition $\tbfx_0 = \bfV_{\nr}^T\bfx_0$ and inputs $\bfu_0, \dots, \bfu_{K-1}$.
\label{prop:rX=tX}
\end{proposition}
\begin{proof}
With zero padding, the operators $\bfA_1, \dots, \bfA_{\ell}$ of \eqref{eq:Prelim:FOMPolynomial} can be reformulated to $\sfA_1, \dots, \sfA_{\ell}$ so that
\[
\bfx_{k + 1} = \sum_{i = 1}^{\ell} \sfA_i (\underbrace{\bfx_k \otimes \dots \otimes \bfx_k}_{i-\text{times}}) + \bfB\bfu_k\,,\qquad k = 0, \dots, K-1\,,
\]
where $\otimes$ is the Kronecker product. Note that $\sfA_i \in \mathbb{R}^{\nh \times \nh^i}$ for $i = 1, \dots, \ell$. Similarly, operators $\tsfA_1, \dots, \tsfA_{\ell}$ are obtained via projection
\[
\tsfA_i = \bfV_{\nr}^T\sfA_i (\underbrace{\bfV_{\nr} \otimes \dots \otimes \bfV_{\nr}}_{i-\text{times}})\,,\qquad i = 1, \dots, \ell\,,
\]
so that $\tbfA_i\tbfx_k^i = \tsfA_i(\tbfx_k \otimes \dots \otimes \tbfx_k)$ holds for $i = 1, \dots, \ell$. Set $\rbfx_0 = \bfV_{\nr}^T\bfx_0$ and note that $\bfx_0 \in \Vcal_{\nr}$ and thus $\rbfx_0 = \tbfx_0$. Querying system \eqref{eq:Prelim:FOMPolynomial} at initial condition $\bfV_{\nr}\rbfx_0 = \bfx_0$ as described in line~\ref{alg:ReProj:Query} of Algorithm~\ref{alg:ReProj} leads to 
\begin{align}
\bfx_{\text{tmp}} = & \sum_{i = 1}^{\ell}\bfA_i\bfx_0^k + \bfB\bfu_0\,,\notag\\
= & \sum_{i = 1}^{\ell}\sfA_i(\bfx_0 \otimes \dots \otimes \bfx_0) + \bfB\bfu_0\notag\\
= & \sum_{i = 1}^{\ell} \sfA_i(\bfV_{\nr}\tbfx_0 \otimes \dots \otimes \bfV_{\nr}\tbfx_0) + \bfB \bfu_0\label{eq:ReProj:rX=tX:One}\\
= & \sum_{i = 1}^{\ell} \sfA_i(\bfV_{\nr} \otimes \dots \otimes \bfV_{\nr})(\tbfx_0 \otimes \dots \otimes \tbfx_0) + \bfB\bfu_0\label{eq:ReProj:rX=tX:Two}\,,
\end{align}
where we used that $\bfx_0 = \bfV_{\nr}\tbfx_0$ in \eqref{eq:ReProj:rX=tX:One} and where we exploited the mixed-product property of the Kronecker product in \eqref{eq:ReProj:rX=tX:Two}. We now project \eqref{eq:ReProj:rX=tX:Two} to obtain
\begin{align*}
\bfV_{\nr}^T\bfx_{\text{tmp}} = & \sum_{i = 1}^{\ell}\bfV_{\nr}^T\sfA_i(\bfV_{\nr} \otimes \dots \otimes \bfV_{\nr})(\tbfx_0 \otimes \dots \otimes \tbfx_0) + \bfV_{\nr}^T\bfB\bfu_0\\
= & \sum_{i = 1}^{\ell} \tsfA_i(\tbfx_0 \otimes \dots \otimes \tbfx_0) + \tbfB\bfu_0
\end{align*}
and thus $\tbfx_1 = \bfV_{\nr}^T\bfx_{\text{tmp}}$. According to line~\ref{alg:ReProj:ReProj} in Algorithm~\ref{alg:ReProj}, the re-projected state is $\rbfx_1 = \bfV_{\nr}^T\bfx_{\text{tmp}}$ and thus $\tbfx_1 = \rbfx_1$ holds. The same steps can be repeated for time step $k$ with $\rbfx_k = \tbfx_k$ to obtain $\tbfx_{k + 1} = \rbfx_{k + 1}$. Then, with induction follows that $\rbfX = \tbfX$ and $\rbfY = \tbfY$ hold. 
\end{proof}

\begin{corollary}
Let the trajectories $\rbfX$ and $\rbfY$ of length $K$ be generated with Algorithm~\ref{alg:ReProj} from a system with polynomial nonlinear terms up to degree $\ell$. Let further  
\begin{equation}
K \geq p + \sum_{i = 1}^{\ell} n_i\,,
\label{eq:ReProj:Cond}
\end{equation}
with $n_i$ defined in \eqref{eq:Prelim:nri} for $i = 1, \dots, \ell$. Consider the data matrix 
\begin{equation}
\rbfD = \begin{bmatrix}
\rbfX\\
\rbfX^2\\
\vdots\\
\rbfX^j\\
\bfU
\end{bmatrix} \in \mathbb{R}^{\sum_{i = 1}^{\ell} n_i + p \times K}
\label{eq:ReProj:DataMatrixReProj}
\end{equation}
derived from the re-projected trajectory $\rbfX$, cf.~the data matrix $\bbfD$ derived from the projected trajectory $\bbfX$ defined in \eqref{eq:Prelim:OpInfDataMatrixProj}. If $\rbfD$ has full rank, then the least-squares problem 
\begin{equation}
\min_{\hbfO} \|\rbfD^T\hbfO^T - \rbfY^T\|_F^2
\label{eq:ReProj:LSQReProj}
\end{equation}
has a unique solution $\hbfO^*$ with objective 0 and that solution is $\hbfO^* = [\tbfA_1, \tbfA_2, \dots, \tbfA_{\ell}, \tbfB]$, where $\tbfA_1, \dots, \tbfA_{\ell}, \tbfB$ are the intrusive operators \eqref{eq:Prelim:ROMOperators}.
\label{cor:Exact}
\end{corollary}
\begin{proof}
First, because of Proposition~\ref{prop:rX=tX}, we have $\tbfX = \rbfX$ and $\tbfY = \rbfY$, and thus the states of $\rbfX$ and $\rbfY$ satisfy the equations corresponding to the intrusive reduced model \eqref{eq:Prelim:iROM}. This means that the matrix $\tbfO = [\tbfA_1, \tbfA_2, \dots, \tbfA_{\ell}, \tbfB]$ is a solution of \eqref{eq:ReProj:LSQReProj} because it achieves objective 0. To show uniqueness, note that \eqref{eq:ReProj:LSQReProj} corresponds to $i = 1, \dots, \nr$ independent least-squares problems
\begin{equation}
\min_{\hbfO_i} \|\rbfD\hbfO_i^T - \rbfY_i^T\|_2^2\,,
\label{eq:ReProj:SeparateOpInf}
\end{equation}
with $\hbfO = [\hbfO_1^T, \dots, \hbfO_{\nr}^T]^T$ and $\rbfY = [\rbfY_1^T, \dots, \rbfY_{\nr}^T]^T$. Each of the rows of $\hbfO$ has length $p + \sum_{i = 1}^{\ell} \nr_i$ and thus each of the least-squares problems \eqref{eq:ReProj:SeparateOpInf} has $p + \sum_{i = 1}^{\ell} \nr_i$ unknowns. Condition \eqref{eq:ReProj:Cond} guarantees that the number of equations in each least-squares problem \eqref{eq:ReProj:SeparateOpInf} is at least $K \geq p + \sum_{i = 1}^{\ell} \nr_i$. Thus, if $\rbfD$ has full rank, then there is at most one solution that solves \eqref{eq:ReProj:SeparateOpInf}. Since $\tbfO$ leads to objective 0, we obtain $\hbfO^* = \tbfO$. 
\end{proof}

\section{Computational procedure and practical aspects}
\label{sec:CompProc}
This section summarizes the overall computational procedure of operator inference with re-projected trajectories in Algorithm~\ref{alg:OpInfWithReProj} and discusses practical aspects as well as limitations of the approach.

\subsection{Computational procedure}
Algorithm~\ref{alg:OpInfWithReProj} summarizes the overall procedure of recovering reduced models from re-projected trajectories with operator inference. The inputs to Algorithm~\ref{alg:OpInfWithReProj} are $\bff$, the degree $\ell$, the dimension $\nr$ of the reduced space, the parameters $\bfmu_1, \dots, \bfmu_m$, the initial conditions $\bfx_0(\bfmu_1), \dots, \bfx_0(\bfmu_m)$, and the input trajectories $\bfU(\bfmu_1), \dots, \bfU(\bfmu_m)$. Algorithm~\ref{alg:OpInfWithReProj} time steps the high-dimensional system to obtain the trajectories $\bfX(\bfmu_1), \dots, \bfX(\bfmu_m)$ in the \texttt{for} loop on line~\ref{alg:OpInfWithReProj:ForInitial}. Then, in line~\ref{alg:OpInfWithReProj:POD}, the POD basis matrix $\bfV_{\nr}$ is computed from the snapshot matrix $[\bfX(\bfmu_1), \dots, \bfX(\bfmu_m)]$. The \texttt{for} loop in line~\ref{alg:OpInfWithReProj:ReProj} calls Algorithm~\ref{alg:ReProj} to generate the re-projected trajectories $\rbfX(\bfmu_1), \dots, \rbfX(\bfmu_m)$ and $\rbfY(\bfmu_1), \dots, \rbfY(\bfmu_m)$. Operator inference as described in Corollary~\ref{cor:Exact} is then applied to the re-projected trajectories in line~\ref{alg:OpInfWithReProj:Assemble} and line~\ref{alg:OpInfWithReProj:Solve} to learn operators. Line~\ref{alg:OpInfWithReProj:Return} returns the learned operators.

The computational costs of Algorithm~\ref{alg:OpInfWithReProj} are typically dominated by querying the high-dimensional system. The costs of assembling the data matrix on line~\ref{alg:OpInfWithReProj:Assemble} and the costs of solving the corresponding least-squares problem on line~\ref{alg:OpInfWithReProj:Solve} typically are negligible. In the \texttt{for} loop in line~\ref{alg:OpInfWithReProj:ForInitial}, the high-dimensional system is time stepped to generate the trajectories for constructing the POD basis matrix, which is similar to traditional, intrusive model reduction. The \texttt{for} loop in line~\ref{alg:OpInfWithReProj:ReProj} requires time stepping the high-dimensional systems once more to sample the re-projected trajectories with Algorithm~\ref{alg:ReProj}. Thus, the computational costs of learning a reduced model with operator inference with re-projection is twice as high as the costs of constructing a model with operator inference without re-projection. Note, however, that it is unnecessary to sample re-projected trajectories of length $K$. Sampling shorter re-projected trajectories can significantly reduce the computational costs of operator inference with re-projection.

\subsection{Practical aspects and condition of least-squares problem}
\label{sec:ReProj:Practical}
We make three remarks of practical aspects of operator inference with re-projection. First, Corollary~\ref{cor:Exact} states that operator inference from re-projected trajectories gives the intrusive reduced models if condition \eqref{eq:ReProj:Cond} is satisfied and if the data matrix $\rbfD$ defined in \eqref{eq:ReProj:DataMatrixReProj} has full rank. It is straightforward to numerically verify these two conditions in practice and so to determine if Corollary~\ref{cor:Exact} applies and if the intrusive reduced model is obtained up to numerical errors.

Second, to sample the re-projected trajectories with Algorithm~\ref{alg:ReProj}, it is necessary to have available the high-dimensional system in the sense that it can be time stepped for a single time step with initial condition $\rbfx_k$ for $k = 0, \dots, K-1$. This is in contrast to operator inference without re-projection, which is applicable even if only the trajectories $\bfX(\bfmu_1), \dots, \bfX(\bfmu_m)$ and the corresponding inputs $\bfU(\bfmu_1), \dots, \bfU(\bfmu_m)$ are available and the high-dimensional system cannot be queried. However, note that it is unnecessary to time step the high-dimensional system at arbitrary initial conditions. The re-projected states are close to the states of the high-dimensional system if the space $\Vcal_{\nr}$ is sufficiently rich, which typically is a necessary requirement for the success of model reduction in any case.

Third, in practice, the condition number of $\rbfD^T(\bfmu_j)\rbfD(\bfmu_j), j = 1, \dots, m$ can be high, which means that numerical errors are amplified and pollute the learned operators even if all conditions required for Corollary~\ref{cor:Exact} are satisfied. To keep the condition number of $\rbfD^T(\bfmu_j)\rbfD(\bfmu_j)$ low, we concatenate multiple trajectories corresponding to different inputs in practice. Let $\bfU_1(\bfmu_j), \dots, \bfU_{m^{\prime}}(\bfmu_j)$ be $m^{\prime} \in \mathbb{N}$ input trajectories and let $\bfX_{1}(\bfmu_j), \dots, \bfX_{m^{\prime}}(\bfmu_j)$ be the corresponding trajectories and $\rbfX_{1}(\bfmu_j), \dots, \rbfX_{m^{\prime}}(\bfmu_j)$ be the corresponding re-projected trajectories computed with Algorithm~\ref{alg:ReProj}. We concatenate the trajectories to obtain
\begin{equation}
\qquad \bfU(\bfmu_j) = \begin{bmatrix}
\bfU_1(\bfmu_j), \dots, \bfU_{m^{\prime}}(\bfmu_j)
\end{bmatrix}\,,\quad \rbfX = \begin{bmatrix}
\rbfX_1(\bfmu_j), \dots, \rbfX_{m^{\prime}}(\bfmu_j)
\end{bmatrix}\,,
\label{eq:ReProj:StackedTraj}
\end{equation}
and then use \eqref{eq:ReProj:StackedTraj} and $\rbfY(\bfmu_j)$ obtained from $\rbfY_1(\bfmu_j), \dots, \rbfY_{m^{\prime}}(\bfmu_j)$ in the least-squares problem \eqref{eq:ReProj:LSQReProj} to learn a model. This is a similar process as suggested in \cite{Peherstorfer16DataDriven}.

\begin{algorithm}[t]
\caption{Operator inference with re-projected trajectories}\label{alg:OpInfWithReProj}
\begin{algorithmic}[1]
\Procedure{OpInfRP}{$\bff, \ell, \nr, \bfmu_1, \dots, \bfmu_m, \bfx_0(\bfmu_1), \dots, \bfx_0(\bfmu_m), \bfU(\bfmu_1), \dots, \bfU(\bfmu_m)$}
\For {$j = 1, \dots, m$}\label{alg:OpInfWithReProj:ForInitial}
\State Time step $\bff$ at $\bfmu_j$ with $\bfx_0(\bfmu_j)$ and $\bfU(\bfmu_j)$ to obtain $\bfX(\bfmu_j)$
\EndFor
\State Derive POD basis matrix $\bfV_{\nr}$ from snapshot matrix $[\bfX(\bfmu_1), \dots, \bfX(\bfmu_m)]$\label{alg:OpInfWithReProj:POD}
\For {$j = 1, \dots, m$}\label{alg:OpInfWithReProj:ReProj}
\State Call Algorithm~\ref{alg:ReProj} with $\bfV_{\nr}, \bfx_0(\bfmu_j), \bfU(\bfmu_j)$ to obtain $\rbfX(\bfmu_j)$ and $\rbfY(\bfmu_j)$ 
\State Assemble data matrix $\rbfD(\bfmu_j)$ defined in \eqref{eq:ReProj:DataMatrixReProj}\label{alg:OpInfWithReProj:Assemble}
\State Solve \eqref{eq:ReProj:LSQReProj} to learn operators $\hbfA_1(\bfmu_j), \dots, \hbfA_{\ell}(\bfmu_j), \hbfB(\bfmu_j)$\label{alg:OpInfWithReProj:Solve}
\EndFor
\State Return learned operators $\hbfA_{1}(\bfmu_j), \dots, \hbfA_{\ell}(\bfmu_j), \hbfB(\bfmu_j)$ for $j = 1, \dots, m$\label{alg:OpInfWithReProj:Return}
\EndProcedure
\end{algorithmic}
\end{algorithm}

\section{Numerical results}
\label{sec:NumExp}
The numerical results in this section demonstrate that the proposed data sampling strategy with re-projection leads to low-dimensional models that match reduced models derived with traditional model reduction methods up to numerical errors in practice. The toy example introduced in the problem formulation in Section~\ref{sec:Prelim:ProbForm} is revisited in Section~\ref{sec:NumExp:ToyExample}. Section~\ref{sec:NumExp:Burgers} derives models for the viscous Burgers' equation and Section~\ref{sec:NumExp:Chafee} for the Chafee-Infante equation. Both of these examples are one dimensional in the spatial domain. Section~\ref{sec:NumExp:DiffReact} demonstrates learning models from re-projected trajectories on a diffusion-reaction equation with two spatial dimensions. 

\begin{figure}
\centering
\begin{tabular}{cc}
\multicolumn{2}{c}{{\resizebox{0.48\columnwidth}{!}{\LARGE\input{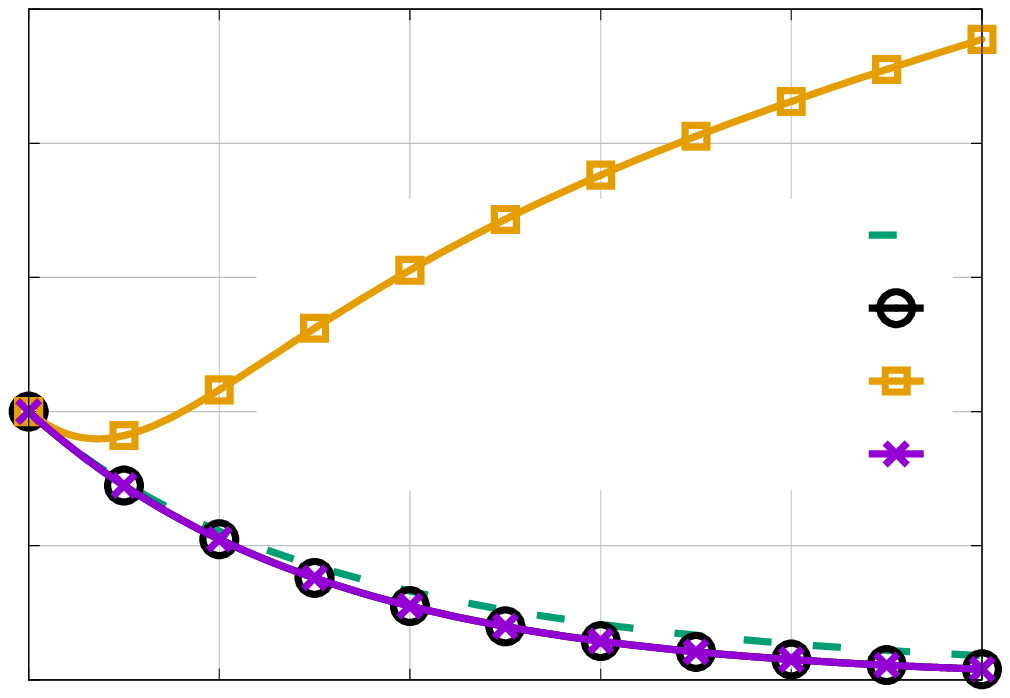}}}}  \\
\multicolumn{2}{c}{\scriptsize (a) trajectories}\\
{\resizebox{0.48\columnwidth}{!}{\LARGE\input{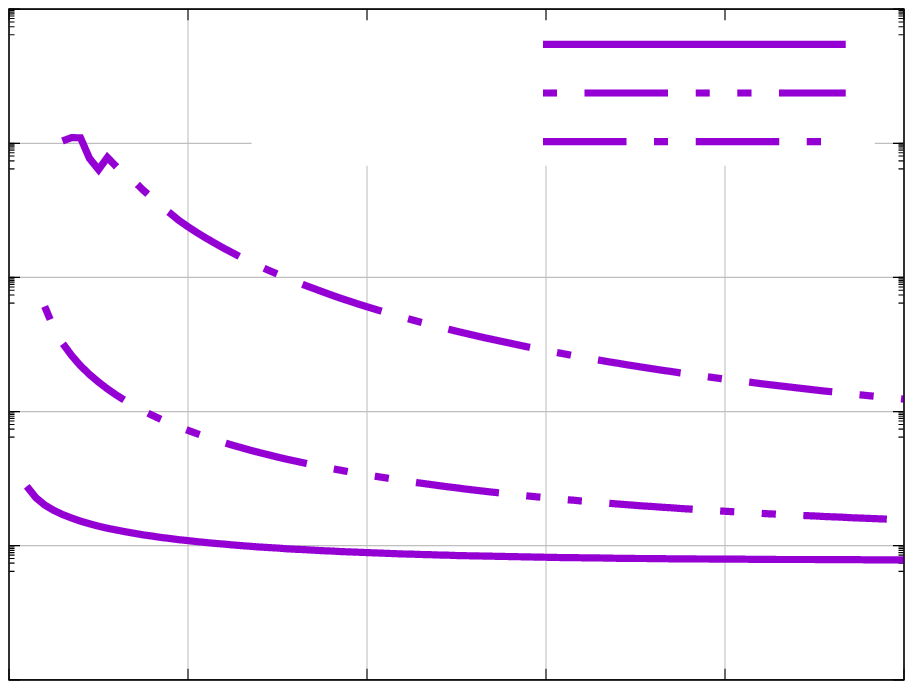}}} & 
{\resizebox{0.48\columnwidth}{!}{\LARGE\input{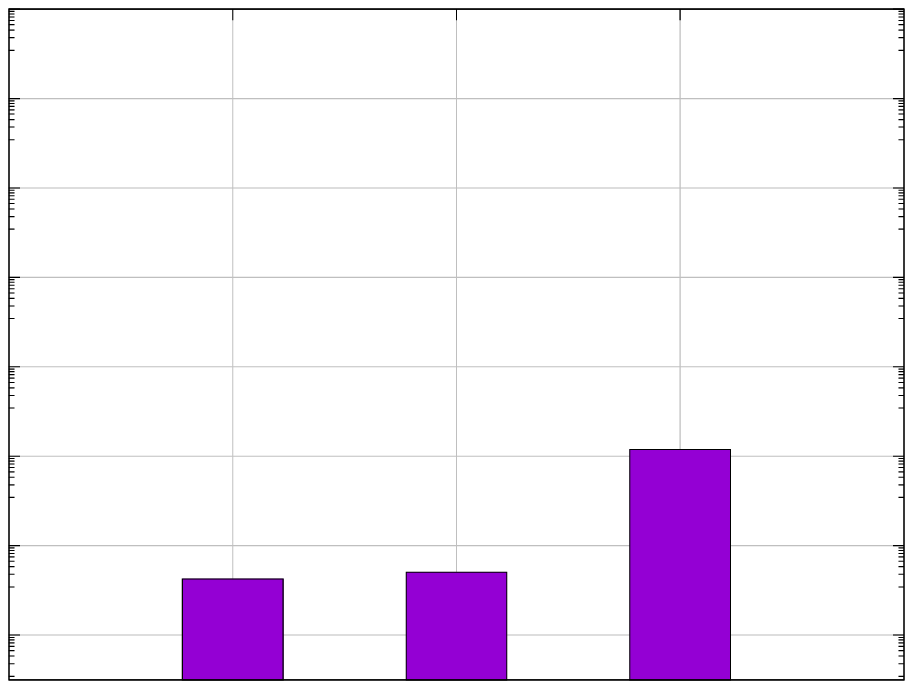}}}\\
\scriptsize (b) condition number of $\rbfD^T\rbfD$ & \scriptsize (c) difference of reduced model and learned model
\end{tabular}
\caption{Toy example: Plot (a) shows that time stepping the model fitted to re-projected trajectories gives a trajectory that matches the trajectory obtained with the intrusive reduced model. Plot (b) and (c) show that the condition number of $\rbfD^T\rbfD$ can be high, which means that numerical errors are amplified. Increasing the number of time steps $K$ and concatenating multiple trajectories as described in Section~\ref{sec:ReProj:Practical}, and as used in Sections~\ref{sec:NumExp:Burgers}--\ref{sec:NumExp:DiffReact}, typically helps to keep the condition number reasonably low in practice.}
\label{fig:NumExp:ToyExample}
\end{figure}

\subsection{Toy example}
\label{sec:NumExp:ToyExample}
We revisit the toy example introduced in Section~\ref{sec:Prelim:ProbForm}. Let $\rbfX$ be the re-projected trajectory obtained with Algorithm~\ref{alg:ReProj}. Following the least-squares problem \eqref{eq:ReProj:LSQReProj} described in Corollary~\ref{cor:Exact}, we learn a model from the re-projected trajectory $\rbfX$ and time step the learned model to obtain the trajectory $\hbfX$, which is plotted in Figure~\ref{fig:NumExp:ToyExample}a. The trajectory of the model learned from the re-projected trajectory closely follows the trajectory of the intrusive reduced model, which is in stark contrast to the model learned from the trajectory $\bbfX$ without re-projection. Thus, the results in Figure~\ref{fig:NumExp:ToyExample}a are in agreement with Corollary~\ref{cor:Exact}. 

Now consider the data matrix $\rbfD$ defined in \eqref{eq:ReProj:DataMatrixReProj}. Figure~\ref{fig:NumExp:ToyExample}b shows the condition number of $\rbfD^T\rbfD$ for dimensions $\nr \in \{2, 4, 6\}$ and various numbers of time steps $K$. In this example, the condition number grows with the dimension $\nr$. This means that even though condition \eqref{eq:ReProj:Cond} together with a full-rank data matrix are sufficient to recover the intrusive reduced model, numerical errors are introduced into the learned operators because of the potentially high condition number of $\rbfD^T\rbfD$; cf.~Section~\ref{sec:ReProj:Practical}. Figure~\ref{fig:NumExp:ToyExample}c demonstrates that the difference
\begin{equation}
\frac{\|\hbfX - \tbfX\|_F}{\|\tbfX\|_F}
\label{eq:NumExp:DiffErrorNoParameter}
\end{equation}
between the trajectory $\tbfX$ of the intrusive reduced model and the trajectory $\hbfX$ of the model learned from the re-projected trajectory grows with the dimension $\nr$ as numerical errors are amplified by the increasing condition number of $\rbfD^T\rbfD$ in this example. Increasing the number of time steps $K$ seems to help to reduce the condition number, as shown in Figure~\ref{fig:NumExp:ToyExample}b.

\subsection{Burgers' equation}
\label{sec:NumExp:Burgers}
A similar setup as in \cite{Peherstorfer16DataDriven} is used for demonstrating the proposed approach on the viscous Burgers' equation.

\subsubsection{Setup}
\label{sec:NumExp:Burgers:Setup}
Set the spatial domain to $\Omega = (-1, 1) \subset \mathbb{R}$ and the parameter domain to $\Dcal = [10^{-1}, 1]$. Let $T = 1$ be end time. Consider the viscous Burgers' equation 
\[
\frac{\partial}{\partial t}x(\xi, t; \mu) + x(\xi, t; \mu)\frac{\partial}{\partial \xi} x(\xi, t; \mu) - \mu \frac{\partial^2}{\partial \xi^2} x(\xi, t; \mu) = 0\,, \qquad \xi \in \Omega\,,
\] 
with the spatial coordinate $\xi \in \Omega$, time $t \in [0, T]$, and parameter $\mu \in \Dcal$. Impose Dirichlet boundary conditions $x(-1, t; \mu) = u(t)$ and $x(1, t; \mu) = -u(t)$ with the input function $u: [0, T] \to \mathbb{R}$. The initial condition is zero $x(\xi, 0; \mu) = 0$ for $\xi \in \Omega$. We discretize the Burgers' equation with finite difference on an equidistant grid in $\Omega$ with mesh width $2^{-7}$, which leads to a system of ordinary differential equations of order $\nh = 128$. Time is discretized with the forward Euler method and time step size $\delta t = 10^{-4}$ to obtain
\begin{equation}
\bfx_{k + 1}(\mu) = \bfA_1(\mu)\bfx_k(\mu) + \bfA_2\bfx^2_k(\mu) + \bfB u_k(\mu)\,,\qquad k = 0, \dots, K-1\,,
\label{eq:NumExp:Burgers:Eq}
\end{equation}
which is a polynomial nonlinear dynamical system \eqref{eq:Prelim:FOMPolynomial} of degree $\ell = 2$ with $\bfA_1(\mu) \in \mathbb{R}^{\nh \times \nh}, \bfA_2 \in \mathbb{R}^{\nh \times \nh_2}$, and input matrix $\bfB \in \mathbb{R}^{\nh \times 1}$. Note that $\bfA_2$ and $\bfB$ are independent of the parameter $\mu$. Note further that $\bff(\bfx_k(\mu), u_k(\mu); \mu) = \bfA_1(\mu)\bfx_k(\mu) + \bfA_2\bfx^2_k(\mu) + \bfB u_k(\mu)$ in this example. The number of time steps is $K = 10^4$. 

To generate trajectories from system \eqref{eq:NumExp:Burgers:Eq} for learning low-dimensional models, consider the $m = 10$ equidistant parameters $\mu_1, \dots, \mu_m \in \Dcal$ in the parameter domain $\Dcal$. Set $m^{\prime} = 5$ and consider the input trajectories $\bfU_l(\mu_j) = [u_{l,0}(\mu_j), \dots, u_{l,K-1}(\mu_j)]$ for $j = 1, \dots, m$ and $l = 1, \dots, m^{\prime}$, where $u_{l,i}(\mu_j)$ is a realization of the random variable with uniform distribution in $[0, 10]$ for $i = 0, \dots, K-1$. Then, we generate $m$ trajectories $\bfX_l(\mu_1), \dots, \bfX_l(\mu_m)$ for $l = 1, \dots, m^{\prime}$ to derive the POD basis matrix $\bfV_{\reprojnr}$ of the POD space $\Vcal_{\reprojnr}$ of dimension $\reprojnr \in \mathbb{N}$ from the snapshot matrix $[\bfX(\mu_1), \dots, \bfX(\mu_m)]$. The trajectories are $\bfX(\mu_i) = [\bfX_1(\mu_i), \dots, \bfX_m^{\prime}(\mu_i)]$ for $i = 1, \dots, m$, cf.~Section~\ref{sec:ReProj:Practical}. The re-projected trajectories $\rbfX(\mu_1), \dots, \rbfX(\mu_m)$, and the corresponding trajectories $\rbfY(\mu_1), \dots, \rbfY(\mu_m)$, are obtained by calling Algorithm~\ref{alg:ReProj} for each parameter $\mu_1, \dots, \mu_m$ and for $l = 1, \dots, m^{\prime}$ and by concatenating the trajectories corresponding to the same parameters as described in Section~\ref{sec:ReProj:Practical}. We learn models $\hbff(\cdot, \cdot, \mu_1), \dots, \hbff(\cdot, \cdot, \mu_m)$ by solving the optimization problem \eqref{eq:ReProj:LSQReProj} stated in Corollary~\ref{cor:Exact} using the re-projected trajectories. We verified numerically that the data matrices have full rank. Condition \eqref{eq:ReProj:Cond} holds as well, and thus Corollary~\ref{cor:Exact} is applicable in this setup, which means that we expect that time stepping the learned model gives a trajectory that matches the corresponding trajectory of the intrusive reduced model up to numerical errors. We construct the intrusive reduced models $\tbff(\cdot, \cdot; \mu_1), \dots, \tbff(\cdot, \cdot; \mu_m)$ and learn models $\bbff(\cdot, \cdot; \mu_1), \dots, \bbff(\cdot, \cdot; \mu_m)$ from the projected trajectories $\bbfX(\mu_1), \dots, \bbfX(\mu_m)$ (with\emph{out} re-projection) as described in Section~\ref{sec:Prelim:OpInf}. The projected trajectories $\bbfX(\mu_1), \dots, \bbfX(\mu_m)$ are obtained by concatenating the trajectories $\bbfX_1(\mu_1), \dots, \bbfX_{m^{\prime}}(\mu_m)$ accordingly. For a parameter $\mu \in \Dcal \setminus \{\mu_1, \dots, \mu_m\}$, model $\hbff(\cdot, \cdot; \mu)$ is derived by component-wise spline interpolation of the operators of the learned models $\hbff(\cdot, \cdot; \mu_1), \dots, \hbff(\cdot, \cdot; \mu_m)$. The same interpolation approach is used for deriving the intrusive reduced model $\tbff(\cdot, \cdot; \mu)$ and the model $\bbff(\cdot, \cdot; \mu)$ learned from trajectories without re-projection for $\mu \in \Dcal \setminus \{\mu_1, \dots, \mu_m\}$. To derive model $\hbff(\cdot, \cdot; \mu)$ for dimension $\nr < \reprojnr$, we truncate the operators of $\hbff(\cdot, \cdot; \mu)$ accordingly, which is the same approach as used in \cite{Peherstorfer16DataDriven}. This means that for $\nr < \reprojnr$, the $\nr \times \nr$ submatrix of $\hbfA_1(\mu) \in \mathbb{R}^{\reprojnr \times \reprojnr}$ of model $\hbff(\cdot, \cdot; \mu)$ is extracted, which corresponds to the first $\nr$ POD modes. A similar process is performed for the input matrix, quadratic terms, and higher-degree nonlinear terms if present. Thus, model $\hbff(\cdot, \cdot; \mu)$ is learned once for dimension $\reprojnr$ and then truncated for $\nr < \reprojnr$. The intrusive reduced model $\tbff(\cdot, \cdot; \mu)$ and model $\bbff(\cdot, \cdot; \mu)$ are truncated the same way for $\nr < \reprojnr$.

\begin{figure}
\centering
\begin{tabular}{cc}
{\resizebox{0.48\columnwidth}{!}{\LARGE\input{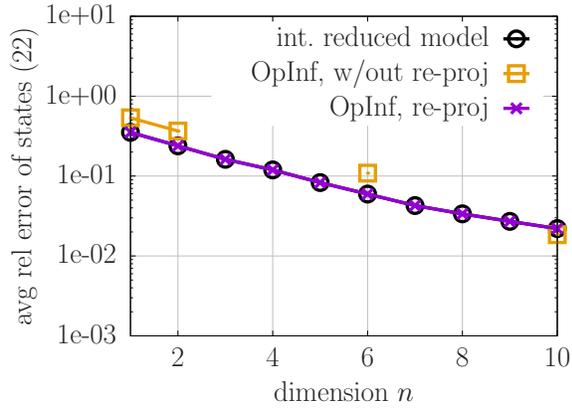}}} & 
{\resizebox{0.48\columnwidth}{!}{\LARGE\input{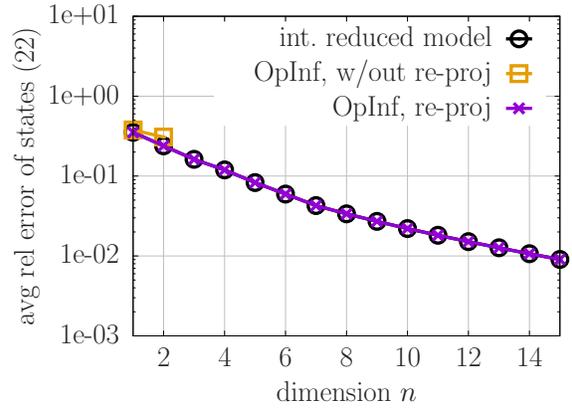}}}\\
\scriptsize (a) training, re-projection dimension $\reprojnr = 10$ & \scriptsize (b) training, re-projection dimension $\reprojnr = 15$\\
{\resizebox{0.48\columnwidth}{!}{\LARGE\input{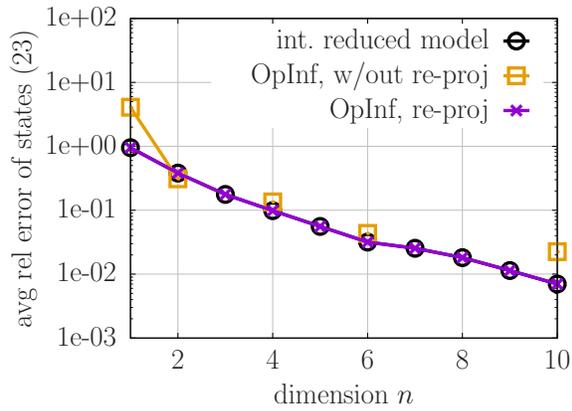}}} & 
{\resizebox{0.48\columnwidth}{!}{\LARGE\input{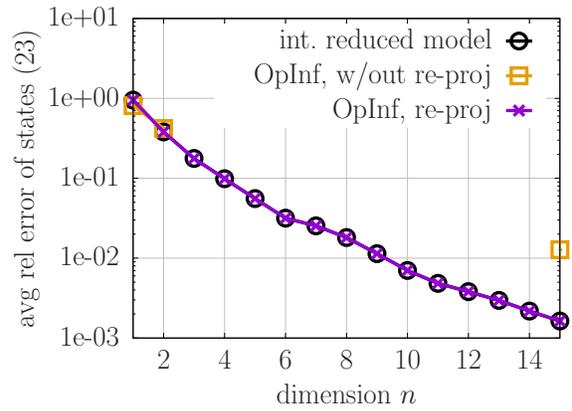}}}\\
\scriptsize (c) test, re-projection dimension $\reprojnr = 10$ & \scriptsize (d) test, re-projection dimension $\reprojnr = 15$
\end{tabular}
\caption{Burgers' equation: The results in plots (a)-(d) demonstrate Corollary~\ref{cor:Exact} that states that models fitted with operator inference to re-projected trajectories are the reduced models obtained with traditional model reduction. In contrast, models fitted to projected trajectories (without re-projection) perform significantly worse and even show unstable behavior (missing values).}
\label{fig:NumExp:BurgersError}
\end{figure}

\subsubsection{Results}
Figure~\ref{fig:NumExp:BurgersError}a shows the error 
\begin{equation}
\frac{1}{m}\sum_{i = 1}^{m} \frac{\|\bfV_{\nr}\bfZ(\mu_i) - \bfX(\mu_i)\|_F}{\|\bfX(\mu_i)\|_F}\,,
\label{eq:NumExp:RelAvgError}
\end{equation}
where $\bfZ(\mu_i) = [\bfZ_1(\mu_i), \dots, \bfZ_{m^{\prime}}(\mu_i)]$, for $i = 1, \dots, m$, is the concatenated trajectory of either the intrusive reduced model $\tbff(\cdot, \cdot; \mu_i)$, the model $\hbff(\cdot, \cdot; \mu_i)$ learned from re-projected trajectories, or the model $\bbff(\cdot, \cdot; \mu_i)$ learned from trajectories without re-projection for all $m^{\prime}$ inputs $\bfU_1(\mu_i), \dots, \bfU_{m^{\prime}}(\mu_i)$. The dimension $\reprojnr$ of the POD space used for re-projection is set to $\reprojnr = 10$ and operators are truncated as described in Section~\ref{sec:NumExp:Burgers:Setup} to compute error \eqref{eq:NumExp:RelAvgError} corresponding to models with $\nr < \reprojnr$. The results in Figure~\ref{fig:NumExp:BurgersError}a are reported for the training parameters $\mu_1, \dots, \mu_m$ and the training inputs that are also used in Section~\ref{sec:NumExp:Burgers:Setup} to construct the POD basis matrix and to learn the models. The intrusive reduced model achieves an error of almost $10^{-2}$ for $\nr = 10$ dimensions. The model learned from trajectories without re-projection exhibits unstable behavior for most dimensions $\nr = 1, \dots, 10$ in the sense that the state during time stepping numerically diverges to NaNs (not a number). Missing values in Figure~\ref{fig:NumExp:BurgersError}a mean that the states diverged to NaNs. In contrast, the model learned from trajectories with re-projection achieves an error \eqref{eq:NumExp:RelAvgError} that closely follows the error of the intrusive reduced model. To test the learned models on parameters that are different from the parameters used for learning the models, we select $m_{\text{test}} = 7$ test parameters $\mu_1^{\text{test}}, \dots, \mu_7^{\text{test}}$ equidistantly in $\Dcal$ and set the input constant to 1. The corresponding error
\begin{equation}
\frac{1}{m_{\text{test}}}\sum_{i = 1}^{m_{\text{test}}} \frac{\|\bfV_{\nr}\bfZ(\mu_i^{\text{test}}) - \bfX(\mu_i^{\text{test}})\|_F}{\|\bfX(\mu_i^{\text{test}})\|_F}\,,
\label{eq:NumExp:RelAvgErrorTest}
\end{equation}
is plotted in Figure~\ref{fig:NumExp:BurgersError}c. The models learned from re-projected trajectories achieve similar errors as the intrusive reduced models, in contrast to models learned from trajectories without re-projection. Similar observations can be made for $\reprojnr = 15$ as shown in Figure~\ref{fig:NumExp:BurgersError}b for training parameters and training inputs and in Figure~\ref{fig:NumExp:BurgersError}d for test parameters and test inputs.

Now consider the difference 
\begin{equation}
\frac{1}{m_{\text{test}}}\sum_{i = 1}^{m_{\text{test}}}\frac{\|\bfZ(\mu_i^{\text{test}}) - \tbfX(\mu_i^{\text{test}})\|_F}{\|\tbfX(\mu_i^{\text{test}})\|_F}
\label{eq:NumExp:DiffError}
\end{equation}
between the trajectories of the intrusive reduced models and the trajectories computed with the learned models. Thus, $\bfZ(\mu_i^{\text{test}})$ in \eqref{eq:NumExp:DiffError} is either the trajectory obtain with $\hbff(\cdot, \cdot; \mu_i^{\text{test}})$ or with $\bbff(\cdot, \cdot; \mu_i^{\text{test}})$ for $i = 1, \dots, m_{\text{test}}$. The difference \eqref{eq:NumExp:DiffError} is plotted in Figure~\ref{fig:NumExp:BurgersDifference}. The models learned from re-projected trajectories achieve a difference to the intrusive reduced model of less than $10^{-10}$, whereas the models learned from trajectories without re-projection are up to 8 orders of magnitude worse in terms of difference \eqref{eq:NumExp:DiffError} and diverge in most cases (missing values in the plots). 

\begin{figure}
\centering
\begin{tabular}{cc}
{\resizebox{0.48\columnwidth}{!}{\LARGE\input{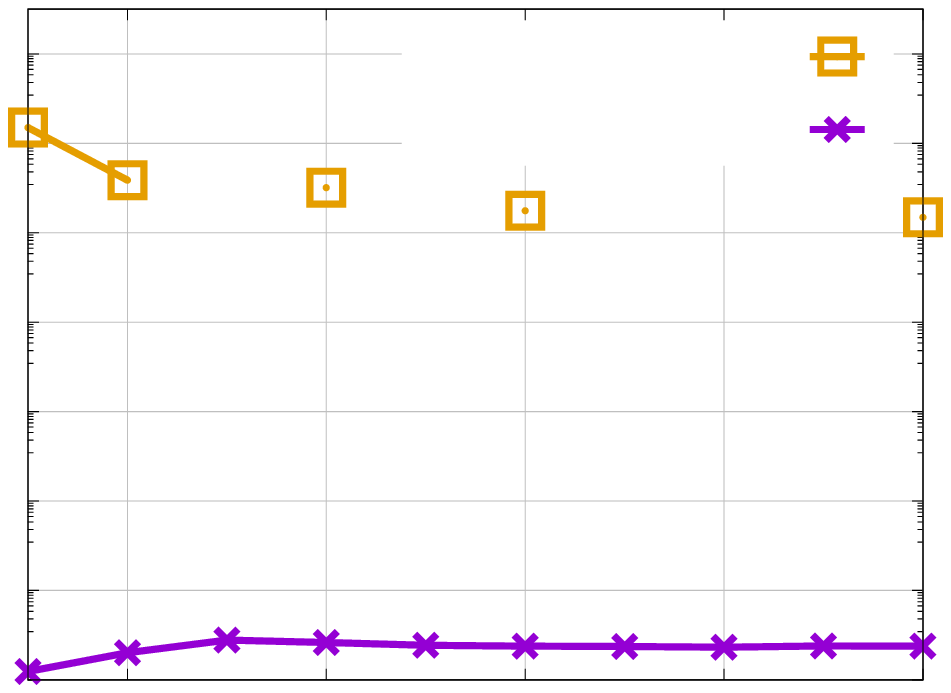}}} & 
{\resizebox{0.48\columnwidth}{!}{\LARGE\input{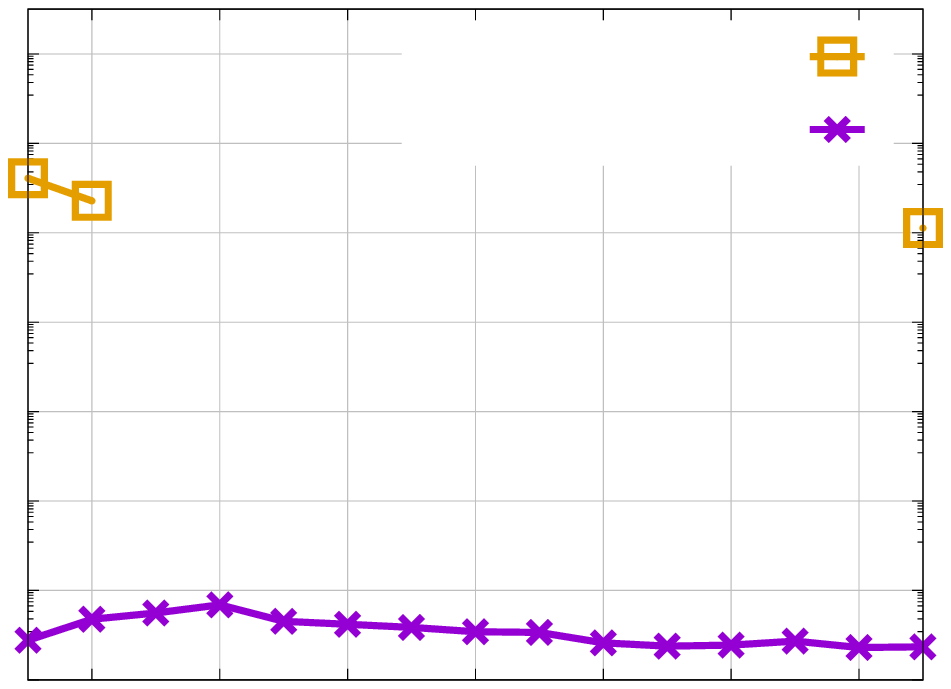}}}\\
\scriptsize (a) difference, re-projection dimension $\reprojnr = 10$ & \scriptsize (b) difference, re-projection dimension $\reprojnr = 15$\\
\end{tabular}
\caption{Burgers' equation: The plots show that time stepping models learned from re-projected trajectories give the same trajectories, up to numerical errors, as intrusive reduced models. In contrast, models learned from trajectories without re-projection lead to significantly different behavior than the corresponding intrusive reduced models. The results shown are for the test parameters $\mu^{\prime}_1, \dots, \mu^{\prime}_7$. Missing values mean that the states of the corresponding model diverged to NaNs during time stepping.}
\label{fig:NumExp:BurgersDifference}
\end{figure}

\subsection{Chafee-Infante equation}
\label{sec:NumExp:Chafee}
A similar setup as in \cite{doi:10.1137/16M1098280} is used in this section. 
\subsubsection{Setup}
Set the spatial domain to $\Omega = (0, 1) \subset \mathbb{R}$ and end time to $T = 4$. We consider the Chafee-Infante equation given by
\[
\frac{\partial}{\partial t} x(\xi, t) - \frac{\partial^2}{\partial \xi^2}x(\xi, t) + x^3(\xi, t) - x(\xi, t) = 0\,,\qquad \xi \in \Omega\,,
\]
with the spatial coordinate $\xi \in \Omega$ and time $t \in [0, T]$. Note that we consider a parameter-independent version of the Chafee-Infante equation. The boundary conditions are
\[
\frac{\partial}{\partial \xi} x(1, t) = 0\,,\qquad x(0, t) = u(t)\,,\qquad t \in [0, T]\,,
\]
with the input $u: [0, T] \to \mathbb{R}$. The initial condition is $x(\xi, t) = 0$ for $\xi \in \Omega \cup \{0, 1\}$. The spatial domain $\Omega$ is discretized on an equidistant grid with mesh width $2^{-7}$ and finite differences. Time is discretized with the forward Euler method and time-step size $\delta t = 10^{-5}$ to obtain the time-discrete dynamical system with polynomial nonlinear terms up to degree $\ell = 3$
\begin{equation}
\bfx_{k + 1} = \bfA_1\bfx_k + \bfA_2\bfx_k^2 + \bfA_3\bfx_k^3 + \bfB u_k\,,\qquad k = 0, \dots, K - 1
\label{eq:NumExp:Chafee:FOM}
\end{equation}
for $K = 4 \times 10^5$ and $\nh = 128$ and where the input matrix $\bfB$ corresponds to the discretization of the boundary conditions.

Consider the $m^{\prime} = 25$ input trajectories $\bfU_1, \dots, \bfU_{m^{\prime}}$ with components sampled from a uniform distribution in $[0, 10]$ and let $\bfX_1, \dots, \bfX_{m^{\prime}}$ be the corresponding trajectories of system \eqref{eq:NumExp:Chafee:FOM}. The same steps as in Section~\ref{sec:NumExp:Burgers:Setup} are performed to concatenate the trajectories $\bfX_1, \dots, \bfX_{m^{\prime}}$, to derive a POD space of dimension $\reprojnr \in \mathbb{N}$ and the corresponding re-projected trajectories $\rbfX_1, \dots, \rbfX_{m^{\prime}}$ and the concatenated re-projected trajectory $\rbfX$, and to learn the model $\hbff(\cdot, \cdot)$ from the re-projected trajectory $\rbfX$. Additionally, as described in Section~\ref{sec:NumExp:Burgers:Setup}, the intrusive reduced model $\tbff(\cdot, \cdot)$ and the model $\bbff(\cdot, \cdot)$ learned from the trajectories without re-projection are constructed. The test input is $u(t) = 25(\sin(\pi t) + 1)$, which is also used in \cite{doi:10.1137/16M1098280}. 

\begin{figure}
\centering
\begin{tabular}{cc}
{\resizebox{0.48\columnwidth}{!}{\LARGE\input{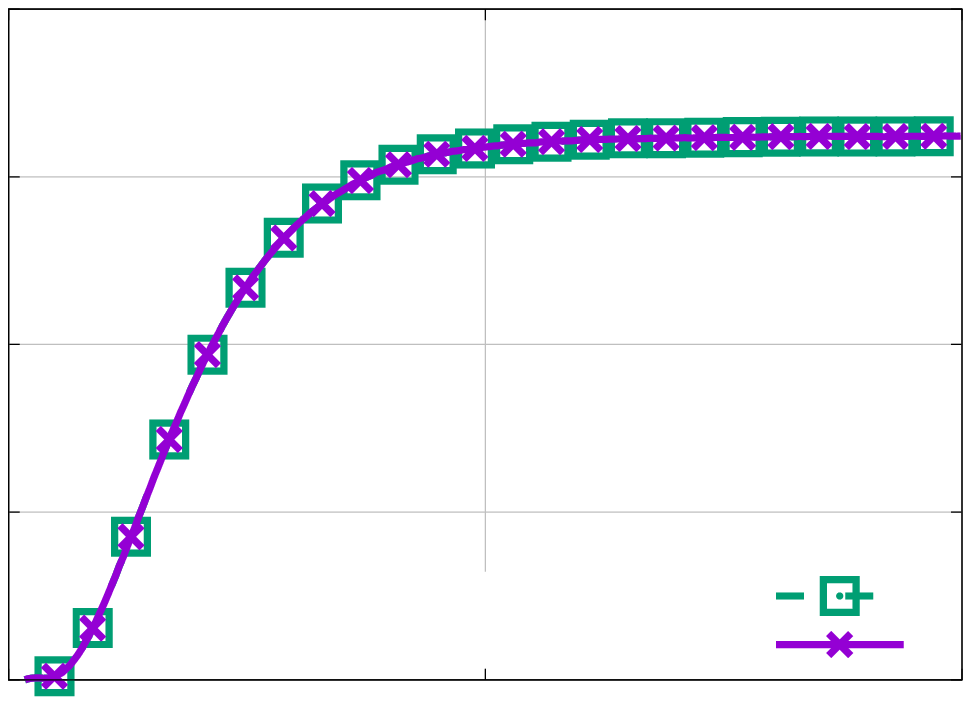}}} & 
{\resizebox{0.48\columnwidth}{!}{\LARGE\input{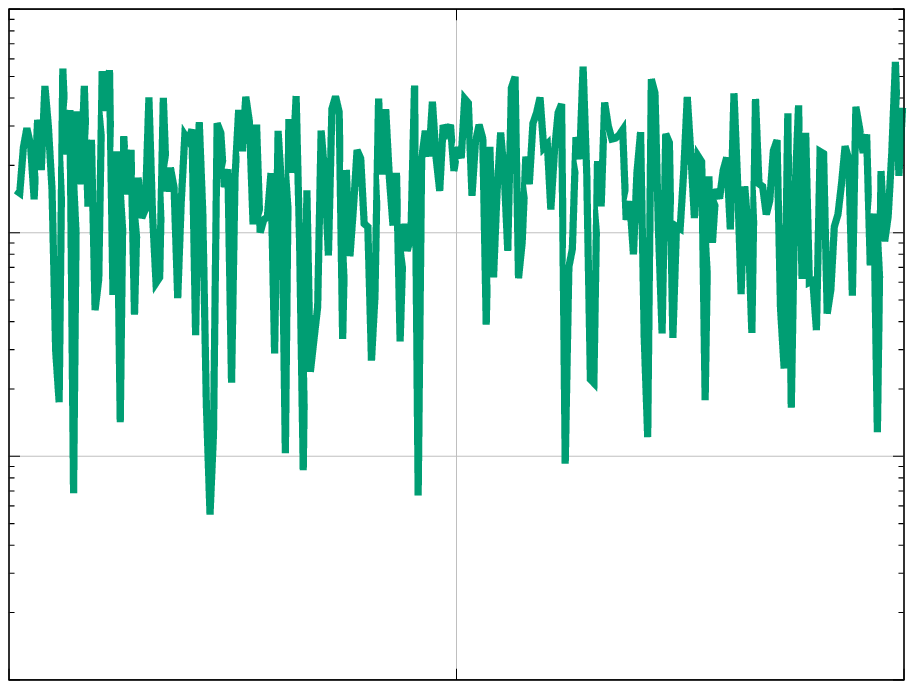}}}\\
\scriptsize (a) trajectories & \scriptsize (b) difference re-projected and projected trajectory
\end{tabular}
\caption{Chafee-Infante: Even though the projected trajectory (without re-projection) and the re-projected trajectory are similar in this example, as shown in plot (a) and (b), the corresponding closure error \eqref{eq:Prelim:ClosureError} has a significant polluting effect on operators learned from trajectories without re-projection as shown in~Figure~\ref{fig:NumExp:ChafeeInfante}.}
\label{fig:NumExp:ChafeeTraj}
\end{figure}

\subsubsection{Results}
Consider the re-projected trajectory $\rbfX_1$ and the projected trajectory $\bbfX_1 = \bfV_{\reprojnr}^T\bfX_1$ for $\reprojnr = 10$. Let $\bfv_{\cdot, \nh} \in \mathbb{R}^{1 \times \reprojnr}$ be the last row of $\bfV_{\reprojnr}$ so that $\bfv_{\cdot, \nh}\bbfX_1$ and $\bfv_{\cdot, \nh}\rbfX_1$ is the approximation of the state at spatial coordinate $\xi = 1$ given by the projected and the re-projected trajectory, respectively. Figure~\ref{fig:NumExp:ChafeeTraj}a plots $\bfv_{\cdot, \nh}\bbfX_1$ and $\bfv_{\cdot, \nh}\rbfX_1$ restricted to time $t \in [0, 1]$. Both trajectories overlap, which indicates that the projected and the re-projected trajectory are similar in this example. The absolute value of the difference $\bfv_{\cdot, \nh}\bbfX_1 - \bfv_{\cdot, \nh}\rbfX_1$ against the time step is shown in Figure~\ref{fig:NumExp:ChafeeTraj}b and indicates again that the projected and the re-projected trajectories are close relative to the absolute value of the trajectories in Figure~\ref{fig:NumExp:ChafeeTraj}a; however, even this small difference has a polluting effect on operator inference that can lead to poor models. Consider Figure~\ref{fig:NumExp:ChafeeInfante}, which shows the test error
\begin{equation}
\frac{\|\bfV_{\nr}\bfZ_{\text{test}} - \bfX_{\text{test}}\|_F}{\|\bfX_{\text{test}}\|_F}\,,
\label{eq:NumExp:Chafee:TrainTestError}
\end{equation}
for $\nr \leq \reprojnr$ and where $\bfZ_{\text{test}}$ is computed with the test input with either model $\bff(\cdot, \cdot), \hbff(\cdot, \cdot), \tbff(\cdot, \cdot)$, or $\bbff(\cdot, \cdot)$. Even though the difference between the projected and the re-projected trajectories is small in this example, the results in Figure~\ref{fig:NumExp:ChafeeInfante} demonstrate that fitting a model to trajectories without re-projection leads to poor approximations of the intrusive reduced models. 

\begin{figure}
\centering
\begin{tabular}{cc}
{\resizebox{0.48\columnwidth}{!}{\LARGE\input{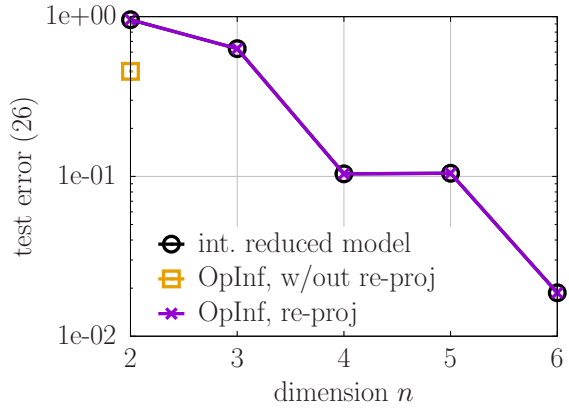}}} & 
{\resizebox{0.48\columnwidth}{!}{\LARGE\input{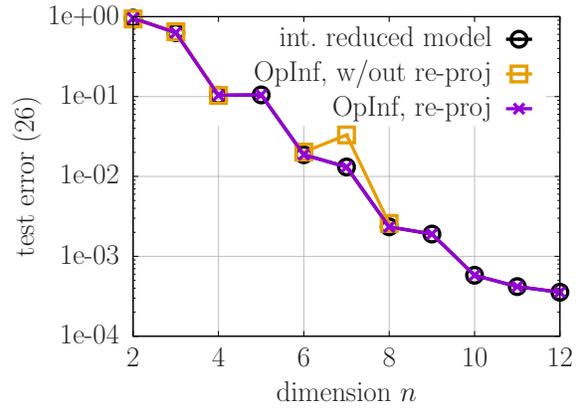}}}\\
\scriptsize (a) test, re-projection dimension $\reprojnr = 6$ & \scriptsize (b) test, re-projection dimension $\reprojnr = 12$\\
{\resizebox{0.48\columnwidth}{!}{\LARGE\input{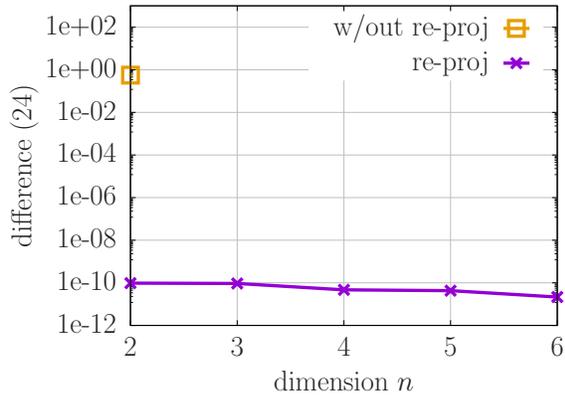}}} & 
{\resizebox{0.48\columnwidth}{!}{\LARGE\input{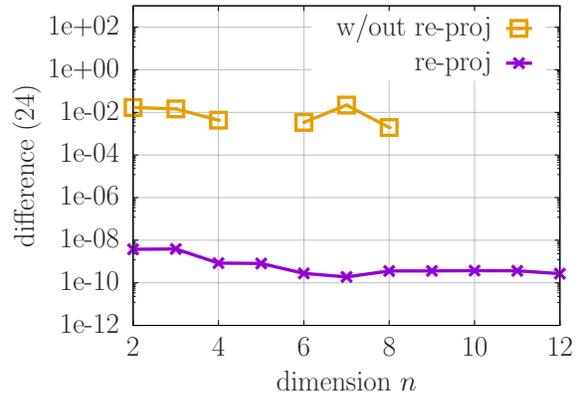}}}\\
\scriptsize (c) test, re-projection dimension $\reprojnr = 6$ & \scriptsize (d) test, re-projection dimension $\reprojnr = 12$
\end{tabular}
\caption{Chafee-Infante equation: Models learned from re-projected trajectories achieve similar performance in terms of error \eqref{eq:NumExp:Chafee:TrainTestError} as intrusive reduced models in this example. Even though the difference between re-projected trajectories and projected trajectories is small in this example (cf.~Figure~\ref{fig:NumExp:ChafeeTraj}), models learned from trajectories without re-projection perform significantly worse than models learned from re-projected trajectories. Missing values correspond to models that numerically diverged during time stepping.}
\label{fig:NumExp:ChafeeInfante}
\end{figure}

\subsection{Diffusion-reaction equation}
\label{sec:NumExp:DiffReact}
The setup of the diffusion-reaction equation in this section follows the example in \cite{PM18MultiTM}. 

\subsubsection{Setup}
Let $\Omega = (0, 1)^2 \subset \mathbb{R}^2$ be the spatial domain with boundary $\partial\Omega$ and closed set $\bar{\Omega} = \Omega \cup \partial \Omega$. Let further $\mu \in \Dcal = [1, 1.5]$ be the parameter domain. Consider the PDE
\begin{equation}
\frac{\partial}{\partial t} x(\bfxi, t; \mu) = -\Delta x(\bfxi, t; \mu) + s(\bfxi)u(t) + g(x(\bfxi, t; \mu))\,,\qquad \bfxi \in \Omega\,,
\label{eq:NumExp:DiffReact:Eq}
\end{equation}
where the spatial coordinate is $\bfxi = [\xi_1, \xi_2]^T$, the source term $s: \bar{\Omega} \to \mathbb{R}$ is $s(\bfxi) = 10^{-1}\sin(2\pi\xi_1)\sin(2\pi\xi_2)$, and the nonlinear term $g: \bar{\Omega} \to \mathbb{R}$ is the second-order Taylor approximation of $x \mapsto -(a\sin(\mu)+2)\exp(-(\mu^2)b)\exp(\mu xc)$ about 0 and with $a = 0.1, b = 2.7$ and $c = 1.8$, which is the same nonlinear term as used in \cite{PM18MultiTM}. The initial condition is 0. We impose homogeneous Neumann boundary conditions and discretize with finite difference on a grid with 64 equidistant grid points in each dimension. Time is discretized with the forward Euler method and time step size $\delta t = 10^{-2}$ and end time $T = 100$. The corresponding time-discrete dynamical system is
\[
\bfx_{k + 1}(\mu) = \bfA_1(\mu)\bfx_k(\mu) + \bfA_2(\mu)\bfx_k^2(\mu) + \bfA_3(\mu)\bfx_k^3(\mu) + \bfB u_k\,,\qquad k = 0, \dots, K - 1\,,
\]
for $K = 10^4$. The dimension $\nh$ of the state $\bfx_k$ at time step $k$ is $\nh = 64^2 = 4096$. Plots of $\bfx_K(\mu)$ for $\mu = 1.0625$ and $\mu = 1.4375$ are given in Figure~\ref{fig:NumExp:DiffReactSurface}.

To construct a reduced space, we select $m = 10$ equidistant parameters $\mu_1, \dots, \mu_m \in \Dcal$ and set the inputs to be realizations of the random variables uniformly distributed in $[1, 1000]$. From these trajectories, the basis matrix $\bfV_{\reprojnr}$ with $\reprojnr = 10$ columns is computed with POD. Then, re-projected trajectories are sampled up to time $t = 5$ (instead of end time $T = 100$). For each $\mu_i$, 10 re-projected trajectories with different random inputs are derived, and concatenated together as described in Section~\ref{sec:ReProj:Practical}. The concatenation of trajectories ensures that the data matrix $\rbfD$ has full rank in this example. Models are learned with operator inference from the re-projected trajectories to obtain $\hbff(\cdot, \cdot; \mu_1), \dots, \hbff(\cdot, \cdot; \mu_m)$. The same process is repeated for the trajectories without re-projection to obtain the models $\bbff(\cdot, \cdot; \mu_1), \dots, \bbff(\cdot, \cdot; \mu_m)$. The rest of the setup is the same as in Section~\ref{sec:NumExp:Burgers}. Test parameters are 7 equidistantly chosen parameters in $\Dcal$. Test inputs are realizations of random variables with uniform distribution in $[1, 1000]$.

\begin{figure}
\centering
\begin{tabular}{cc}
{\resizebox{0.48\columnwidth}{!}{\LARGE\input{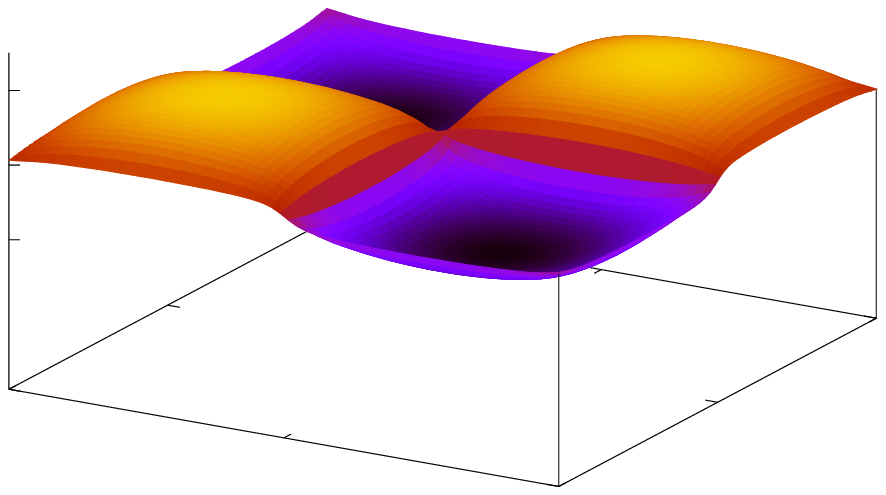}}} & {\resizebox{0.48\columnwidth}{!}{\LARGE\input{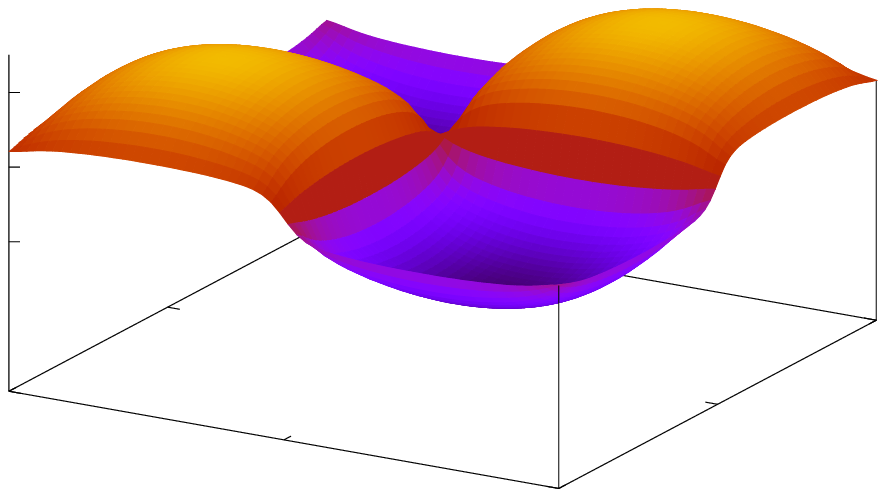}}}\\
\scriptsize (a) high-dimensional system, $\mu = 1.0625$ & \scriptsize (b) high-dimensional system, $\mu = 1.4375$
\end{tabular}
\caption{Diffusion-reaction: Plots show the numerical approximation of the solution of equation \eqref{eq:NumExp:DiffReact:Eq} for parameters $\mu = 1.0625$ and $\mu = 1.4375$, respectively.}
\label{fig:NumExp:DiffReactSurface}
\end{figure}

\subsubsection{Results}
Figure~\ref{fig:NumExp:DiffReactError}a shows the error \eqref{eq:NumExp:RelAvgError} for the training parameters and training inputs. The model fitted to trajectories without re-projection numerically diverged to NaNs during time stepping for all dimensions $\nr > 2$. The model fitted to re-projected trajectories closely matches the behavior of the intrusive reduced model as expected from the analysis presented in Corollary~\ref{cor:Exact}. The same observations can be made for the error \eqref{eq:NumExp:RelAvgErrorTest} with the test parameters and test inputs.

\begin{figure}
\centering
\begin{tabular}{cc}
{\resizebox{0.48\columnwidth}{!}{\LARGE\input{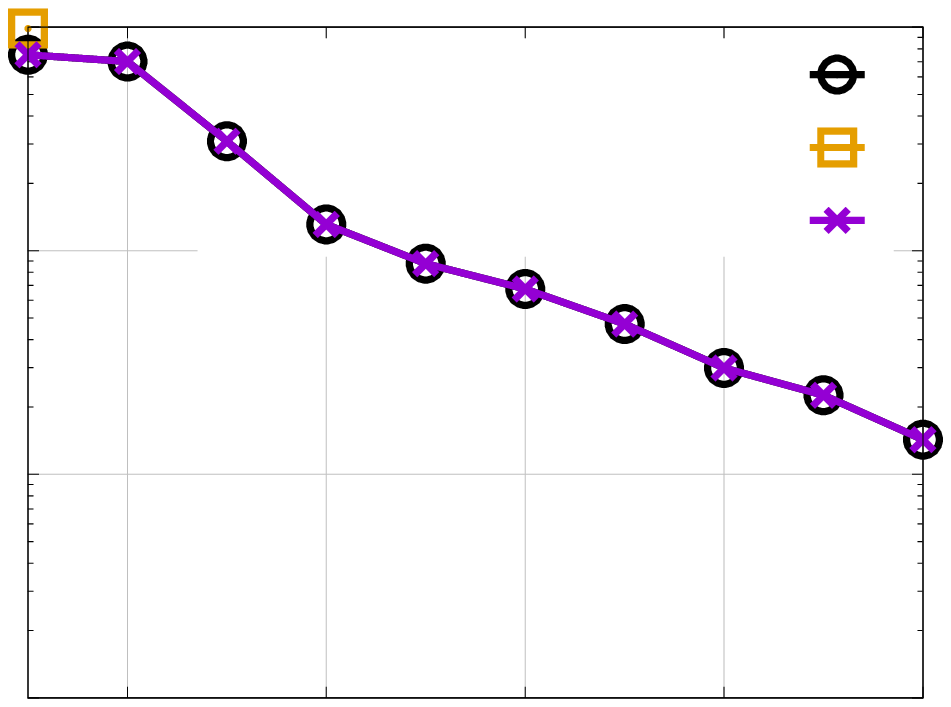}}} & {\resizebox{0.48\columnwidth}{!}{\LARGE\input{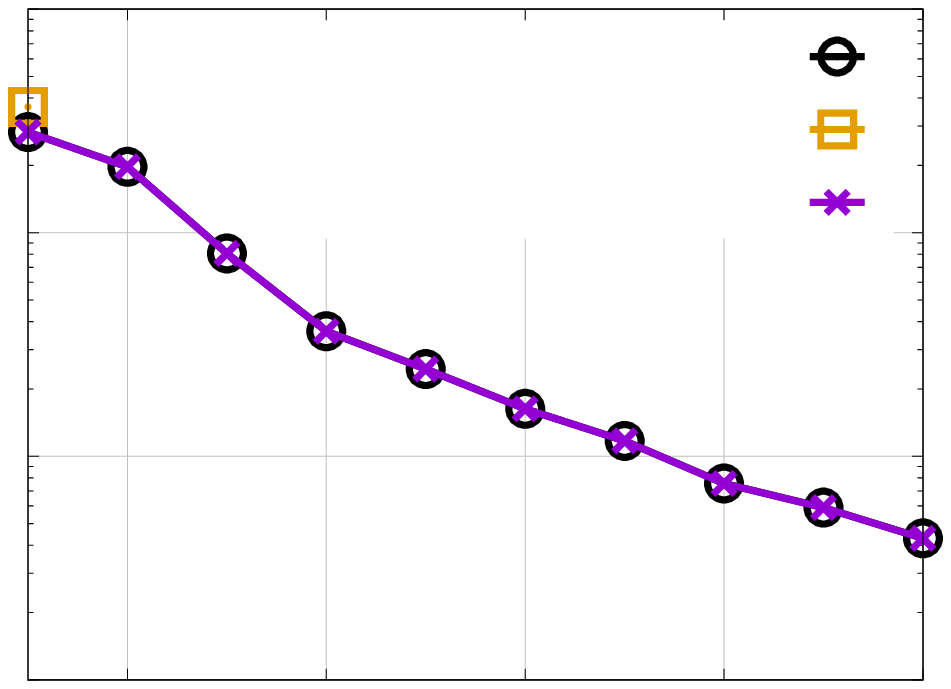}}}\\
\scriptsize (a) training & \scriptsize (b) test
\end{tabular}
\caption{Diffusion-reaction: Models learned from trajectories without re-projection show unstable behavior (missing values) for all dimensions $\nr > 2$. In contrast, models learned with operator inference from re-projected trajectories achieve the same errors as the intrusive reduced models, which is guaranteed by Corollary~\ref{cor:Exact} in this example.}
\label{fig:NumExp:DiffReactError}
\end{figure}

\section{Conclusions}
\label{sec:Conc}
The presented approach exactly recovers reduced models from data under certain conditions. This result holds pre-asymptotically in the number of data points and the dimension of the reduced space as long as the corresponding data matrix is full rank. The optimization problem underlying operator inference with re-projected trajectories is convex and can be solved with standard numerical linear algebra packages. Numerical experiments demonstrate that reduced models are learned up to numerical errors in practice for a wide class of systems with polynomial nonlinear terms. 

\section*{Acknowledgments}
The author would like to thank Elizabeth Qian, Nihar Sawant, and Karen Willcox for many helpful discussions. This work was partially supported by US Department of Energy, Office of Advanced Scientific Computing Research, Applied Mathematics Program (Program Manager Dr. Steven Lee), DOE Award DESC0019334. The numerical experiments were computed with support through the NYU IT High Performance Computing resources, services, and staff expertise. 

\bibliography{reprojopinf}
\bibliographystyle{abbrv}

\end{document}